\documentclass[11pt]{amsart}
\usepackage{amsmath,amssymb,amsthm}
\usepackage{dsfont}
\usepackage{mathrsfs}
\usepackage[margin=2.75cm]{geometry}
\usepackage{graphicx}\usepackage{subfigure}
\usepackage{multicol}
\usepackage{tikz}
\usepackage[all]{xy}
\usepackage{xfrac}
\usepackage{xcolor}
\usepackage{comment}

\usepackage{url}
\usepackage{abstract}

\usepackage[T1]{fontenc}
\usepackage[bookmarksdepth=0]{hyperref}
\hypersetup{   
    colorlinks=true,
    citecolor=blue,
   linkcolor=blue,
    linktoc=all
                         }

\newtheorem{defn}{Definition}[section]
\newtheorem{thm}[defn]{Theorem}
\newtheorem{lem}[defn]{Lemma}
\newtheorem{prop}[defn]{Proposition}
\newtheorem{cor}[defn]{Corollary}
\newtheorem{rmk}[defn]{Remark}

\newcommand{\R}{\mathds{R}}
\newcommand{\Z}{\mathds{Z}}
\newcommand{\x}{\times}
\newcommand{\p}{\partial}
\newcommand{\f}{\frac}
\newcommand{\D}{\mathcal{D}_{>0}}
\newcommand{\K}{\mathds{K}}
\newcommand{\F}{\mathscr{F}}
\newcommand{\G}{\mathscr{G}}
\newcommand{\pr}{\mathcal{P}_R}
\newcommand{\PR}{\mathscr{P}_R}
\newcommand{\PU}{\mathscr{P}_U}
\newcommand{\bu}{\bullet}
\newcommand{\s}{\mathcal{S}}
\newcommand{\U}{\underset}
\newcommand{\C}{{\mathcal{C}_N}}
\newcommand{\q}{\mathbf{q}}

\title{Non-Squeezing Property of Contact Balls}

\author{Sheng-Fu Chiu}
       \address[Sheng-Fu Chiu]{Department of Mathematics, Northwestern University, USA}
       \email{shengfuchiu2009@u.northwestern.edu}

\begin{document} 

\maketitle

\begin{abstract}
In this paper we solve a contact non-squeezing conjecture proposed by Eliashberg, Kim and Polterovich. Let $B_R$ be the open ball of radius $R$ in $\R^{2n}$ and let $\R^{2n}\x\mathds{S}^1$ be the prequantization space equipped with the standard contact structure. Following Tamarkin's idea, we apply microlocal category methods to prove that if $R$ and $r$ satisfy $1\leq\pi r^2<\pi R^2$, then it is impossible to squeeze the contact ball $B_R\x\mathds{S}^1$ into $B_r\x\mathds{S}^1$ via compactly supported contact isotopies.
\end{abstract}

\tableofcontents

\section{Introduction}
\vspace{0.8cm}

\subsection{Contact Non-Squeezability}\mbox{} \\

Since the time Gromov established the celebrated symplectic non-squeezing theorem, there have been attempts to find its analogue in contact topology. The problem is, as pointed out in \cite{EKP} and \cite{Sa}, that the size of an open domain in a contact manifold cannot be recognized by contact embeddings. This is due to the conformal nature of contact structures. Consider the standard contact Euclidean space $(\R^{2n+1}=\{(\q,\mathbf{p},z)\},dz-\q d\mathbf{p})$, the scaling map $(\q,\mathbf{p},z)\mapsto (\lambda\q,\lambda\mathbf{p},\lambda z)$ is a contactomorphism that squeezes any domain into a small neighborhood of the origin when the factor $\lambda>0$ is taken small enough. To avoid the scaling effect, one considers the \textsl{prequantization space} of $\R^{2n}$, that is $\R^{2n}\x\mathds{S}^1$ where $\mathds{S}^1=\R/\Z$, with the contact form $dz+\f{1}{2}(\q d\mathbf{p}-\mathbf{p} d\q)$. The contact ball of radius $R$ is by definition the open domain $B_R\x\mathds{S}^1$. However, this is still not satisfactory. Let $\R^{2n}\cong \mathds{C}^n=\{\vec{\omega}\}$. Pick a positive integer $N$ and define functions $\nu:\mathds{C}^n\rightarrow\R$ and $F_N:\R^{2n}\x\mathds{S}^1\rightarrow\R^{2n}\x\mathds{S}^1 $ by formulas $\nu(\vec{\omega})=\f{1}{\sqrt{1+N\pi\ |\vec{\omega}|^2}}$ and $ F_N(\vec{\omega},z)=(\nu(\vec{\omega})e^{2\pi  iNz}\vec{\omega},z)$. Then we have 

\begin{prop}(c.f. \cite{EKP})
$F_N$ is a contactomorphism that maps $B_R\x\mathds{S}^1$ onto $B_r\x\mathds{S}^1$ with $r=\f{R}{1+NR}$. It turns out that when $N\rightarrow\infty$, we can squeeze any domain into an arbitrary small neighborhood of $\{0\}\x\mathds{S}^1$ by a single contactomorphism.
\end{prop}

In their pioneering paper \cite{EKP}, Eliashberg, Kim, and Polterovich settle the correct notion of contact squeezing, in the sense of the following:

\begin{defn}[Eliashberg-Kim-Polterovich \cite{EKP}]

Let $U_1$, $U_2$ be open domains in a contact manifold $V$. We say $U_1$ can be squeezed into $U_2$ if there exists a compactly supported contact isotopy $\Phi_s:\overline{U_1}\rightarrow V$, $s\in[0,1]$, such that $\Phi_0=Id$, and $\Phi_1(\overline{U_1})\subset U_2$.

\end{defn}

When $V$ is the prequantization space $\R^{2n}\x\mathds{S}^1$, they address the question whether a contact ball of certain size can be squeezed into a smaller one, and then amazingly discover a squeezing and a non-squeezing results for different sizes respectively. The following is their non-squeezing theorem.

\begin{thm}[Eliashberg-Kim-Polterovich \cite{EKP}]
Assume that $n\geq2$. Then for all $0< r,R<\sfrac{1}{\sqrt{\pi}}$, the contact ball $B_R\x\mathds{S}^1$ can be squeezed into the smaller one $B_r\x\mathds{S}^1$.
\end{thm}

The authors offer the following quasi-classical interpretation: think of the length of the circle $\mathds{S}^1$ to be the Planck constant. Let $\R^n\xrightarrow{\rho_0}\R$ be the probability density function and $\R^n\xrightarrow{F_0}\mathds{S}^1$ be the phase function of an unit mass quantum particle at the initial time $t=0$. The quantum motion of this particle in the presence of a potential function $\R^n\xrightarrow{V}\R$ is described by  the Schr\"{o}dinger equation:
$$
\begin{cases}
\dfrac{i}{2\pi}\dfrac{\p\psi}{\p t}=\dfrac{-1}{8\pi}\Delta\psi+ V(\q)\\
\psi_0(\q)=\sqrt{\rho_0(\q)}e^{2\pi i F_0(\q)}  .
\end{cases}
$$

The functions $F_0$ and $V$ generate a Legendrian submanifold of $\R^{2n}\x\mathds{S}^1$:
$$
\mathcal{L}(F_0)=\{(\q,\mathbf{p},z)|p=\f{\p F_0}{\p\q},z=F_0(\q)\}.
$$

On the other hand, the corresponding classical motion of this particle is described by the Hamiltonian flow on the classical phase space $\R^{2n}$ which has a lift to a flow of contactomorphism $\R^{2n}\x\mathds{S}^1\xrightarrow{f_t}\R^{2n}\x\mathds{S}^1$ of the prequantization space given by the following system
$$
\begin{cases}
\dot{\q}=\mathbf{p}\\
\dot{\mathbf{p}}=-\dfrac{\p V}{\p\q}\\
\dot{z}=\f{1}{2}\mathbf{p}^2-V(\q).
\end{cases}
$$ 

For time $t$ small enough, the image $f_t(\mathcal{L}(F_0))$ can be written as a Legendrian submanifold $\mathcal{L}(F_t)$, and the push-forward of density $f_{t*}(\tau^*(\rho_0 d\q))$ can be written as $\tau^*(\rho_t d\q)$. Here $\tau$ denotes the projection from $\R^{2n}\x\mathds{S}^1$ to $\R^{2n}$. The corresponding wave function $\sqrt{\rho_t(\q)}e^{2\pi i F_t(\q)}$ is a quasi-classical approximate solution of the above Schr\"{o}dinger equation. The balls $B_R$ are in fact energy levels of the classical harmonic oscillator. The squeezing theorem says that at  sub-quantum scale (this means that $\pi R^2$ is less than the Planck constant), level sets are indistinguishable. One can also say that the symplectic rigidity of the balls $B_R$ in the classical phase space $\R^{2n}$ is lost after prequantization, i.e. after joining an extra $\mathds{S}^1$-variable. \\

However, the situation is totally different when it comes to large scale. They manage to prove a non-squeezing theorem when an integer can be inserted into the two sizes.

\begin{thm}[Eliashberg-Kim-Polterovich \cite{EKP}]\label{gap}
If there exists an integer $m\in[\pi r^2,\pi R^2]$ , then $B_R\x\mathds{S}^1$ cannot be squeezed into $B_r\x\mathds{S}^1$.
\end{thm}

In \cite{EKP}, they develop the idea of orderability of contactomorphism group and apply this to the study of contact squeezability. For more detailed discussions about orderability we refer the reader to Giroux's Bourbaki paper \cite{Gi} as well as the work of Sandon \cite{Sa}\cite{Sa2} in which she uses generating functions to define a bi-invariant metric on the contactomorphism groups of $\R^{2n}\x\mathds{S}^1$ and recover the non-squeezing theorem of \cite{EKP}. There is also work of Albers-Merry \cite{AlMe} to obtain similar results by using Rabinowitz Floer homology. Also see the work of Chernov and Nemirovski \cite{ChNe1}\cite{ChNe2} for connection to Lorentz geometry.\\

The case $m<\pi r^2<\pi R^2< m+1$ remains open in their paper \cite{EKP}. This is covered by the main result of the present paper:

\begin{thm}\label{main}
If $1\leq \pi r^2<\pi R^2$, then it is impossible to squeeze $B_R\x\mathds{S}^1$ into $B_r\x\mathds{S}^1$.
\end{thm}

\vspace{1.4cm}

\subsection{Microlocal Sheaf Formalism}\ \\

The paper uses terminologies and results from the theory of algebraic microlocal analysis. The major machinery we rely on is about sheaves and their \textsl{microlocal singular supports} which are developed in \cite{KaSch}. Fix a ground field $\K$ and let $D(X)$ denote the derived category of sheaves of $\K-$vector spaces on a smooth manifold $X$. The notion of microlocal singular supports (micro-supports) which we denote by $SS$, is defined as follows:

\begin{defn}(\cite{KaSch}, Definition 5.1.2) Let $\F\in D(X)$ and $p\in T^*X$, we say $p \notin SS(\F)$ if there exists an open neighborhood $U$ of $p$ such that for any $x_0\in X$ and any real $C^1$-function $\phi$ defined in a neighborhood of $x_0$ such that $\phi(x_0)=0$ and $d\phi(x_0)\in U$, we have 
$$
(R\Gamma_{\{x|\phi(x)>0\}}(\F) )_{x_0}\cong 0.
$$
\end{defn}

It is proved in \cite{KaSch} that the set $SS(\F)$ is always conic closed and involutive (also known as coisotropic)  in $T^*X$ and this allows us to approximate a given lagrangian with those $SS(\F)$ satisfying special requirements. In addition, they have certain functorial properties with respect to Grothendieck's six operations. For convenience of the reader, these rules are listed in \textbf{Appendix} (see section \ref{appendix}). Recently, Guillermou \cite{Gu2} deduced a new proof of Gromov-Eliashberg $C^0$-rigidity theorem from the involutivity of micro-supports. See also the paper of Nadler and Zaslow \cite{NaZa} for its application to Fukaya category.

In \cite{Ta}, Tamarkin gives a new approach to symplectic non-displaceability problems based on microlocal sheaf theory. For a closed subset $A$ in $T^*X$ one might consider the category of sheaves micro-supported in $A$, but in general the set $A$ is not conic. Tamarkin's idea is to add  an extra variable $\R_t$ to $X$ and to work in the localized triangulated category $\D(X\x\R_t):=D(X\x\R_t)/D_{\{k\leq0\}}(X\x\R_t)$ where $k$ is the cotangent coordinate of $\R_t$ and $D_{\{k\leq0\}}(X\x\R_t)$ denotes the full subcategory consists of objects $\F$ such that $SS(\F)\subset \{k\leq0\}$. By conification along the extra variable, subsets in $T^*X$ can be lifted to conic ones in $T^*_{k>0}(X\x\R_t)$, which allow us to define a full subcategory $D_A(X\x\R_t)$ of objects in $\D(X\x\R_t)$ micro-supported on the conification of $A$. 

Tamarkin also proves a key theorem of Hamiltonian shifting which asserts that for a given compactly supported Hamiltonian isotopy $\Phi$ of $T^*X$, there exists a functor $\Psi$ from $D_A(X\x\R_t)$ to $D_{\Phi(A)}(X\x\R_t)$ such that any object $\F$ in $D_A(X\x\R_t)$ is isomorphic to its image $\Psi(\F)$ in $\D(X\x\R_t)$ up to torsion objects. Let $T_c:t\mapsto t+c$ be the translation by $c$ units, the torsion object is defined to be the object $\G$ in $\D(X\x\R_t)$ such that the natural morphism $\G\rightarrow T_{c*}\G$ becomes zero in $\D(X\x\R_t)$ when the positive number $c$ is sufficiently large. 

It turns out that contact manifolds and contact isotopies perfectly fit into Tamarkin's conic symplectic framework. Moreover, the possibility of torsion objects is ruled out by the fact that the circle $\mathds{S}^1$ is an additive quotient of the extra variable $\R$. In \textbf{section \ref{manifold}} we transform the problem of contact isotopies to the level of conic $\Z$-equivariant Hamiltonian isotopies on a certain symplectic manifold. 

\textbf{Section \ref{projector}} is devoted to construction of the projector. Let me explain what projector means through this paper. Given a category $\mathscr{D}$ and its full subcategories $\mathscr{C}_1$ and $\mathscr{C}_2$. We say that $\mathscr{C}_1$ is the \textsl{left semi-orthogonal complement} of $\mathscr{C}_2$ if $\mathscr{C}_1=\{\F\in\mathscr{D}| \forall \G\in\mathscr{C}_2, hom_{\mathscr{D}}(\F,\G)=0\} $. A \textsl{projector} from $\mathscr{D}$ to $\mathscr{C}_1$ is a functor $\mathscr{D}\xrightarrow{f_1}\mathscr{C}_1$ which is  right adjoint to the embedding functor $\mathscr{C}_1\xrightarrow{i_1}\mathscr{D}$. Similarly the projector $\mathscr{D}\xrightarrow{f_2}\mathscr{C}_2$ is the left adjoint functor of $\mathscr{C}_2\xrightarrow{i_2}\mathscr{D}$. In other words we have $\mathscr{C}_1\cong \mathscr{D}/\mathscr{C}_2$ and $\mathscr{C}_2\cong \mathscr{D}/\mathscr{C}_1$. Put $E=\R^n$. When $\mathscr{D}=\D(E\x\R_t)$ and $\mathscr{C}_2$ is consisted of objects micro-supported outside the open contact ball (see \ref{action}), the projector $f_2$ is represented explicitly by a convolution kernel $\mathcal{Q}\in \D(E\x E\x\R_t)$. This means that for $\F\in\mathscr{D}$, we have $f_2(\F)\cong\F\underset{t}{\bu}\mathcal{Q}$. There also exist a kernel $\mathcal{P}$ representing $f_1$ and a kernel $\K_\Delta$ representing $Id_{\mathscr{D}}$. Moreover, the natural morphisms of functors $i_1f_1\rightarrow Id_{\mathscr{D}}\rightarrow i_2f_2$ are represented by a distinguished triangle $\mathcal{P}\rightarrow\K_\Delta\rightarrow\mathcal{Q}\xrightarrow{+1}$ in the triangulated category $\D(E\x E\x\R_t)$.

In general these projectors are expected to shrink the possibilities of micro-supports in the cotangent space according to the type of our geometry problems. This allows us to write down the kernels of these projectors. By \textsl{kernel} we mean the sheaf that represents the prescribed functor by composition or convolution with it. The kernel associated to a projector here in this paper is in some sense unique. Tamarkin's idea is to use $Rhom$ between these kernels to produce Hamiltonian  invariants. In his paper \cite{Ta}, $Rhom$ realizes non-displaceability of two compact subsets in the symplectic manifold. As an example, he shows that in $\mathds{CP}^n$ any two choices from the two subsets $\mathds{RP}^n$ and $\mathds{T}^n$ (Clifford torus) are non-displaceable. One can also assign a Novikov ring action on those $Rhom$'s. We hope this formalism will inspire an alternative approach to Fukaya category in the future.

In \textbf{section \ref{invariant}} we construct a family of contact invariants. Given any positive integer $N$ and certain open domains $U$ in $\R^{2n}\x\mathds{S}^1$ we define a complex $\C(U)$ using $Rhom$ between the kernel that shrinks micro-supports into $U$ and the  kernel of discretized circle $\sfrac{\Z}{N\Z}$. With the help of sheaf quantization techniques (see Kashiwara-Schapira-Guillermou \cite{Gu}\cite{GuKaSch}) we manage to prove that these $\C(U)$ are contact isotopy invariants (see \textbf{Theorem \ref{inv}}). The more important thing is that the way we define $\C(U)$ leads to a  natural $\sfrac{\Z}{N\Z}$-action on $\C(U)$ and makes $\C(U)$ an object of the derived category of $\K[\sfrac{\Z}{N\Z}]$-modules. Furthermore this makes $\C(U)$ into an $\sfrac{\Z}{N\Z}$-equivariant object (see \textbf{\ref{cyclic}} for its motivation and explanation). This approach resembles the cyclic homology in a Hamiltonian fashion. We then move forward to the computation of  these invariants when $U$ is a contact ball $B_R\x\mathds{S}^1$ using the explicit form of projectors constructed in the previous section. 

In \textbf{section \ref{proof}} we show that when $\pi r^2$ and $\pi R^2$ are greater than or equal to 1, one can select an appropriate integer $N$ which depends on both $r$ and $R$ such that the naturality of morphisms between the three invariants (i.e., $\C(B_R\x\mathds{S}^1)$ , $\C(B_r\x\mathds{S}^1)$ , $\C(\R^{2n}\x\mathds{S}^1)$) obtained from the presumed squeezability is violated by the graded algebra structure of the Yoneda product $Ext_{\K[\sfrac{\Z}{N\Z}]}^\bu(\K;\K)$ after we pass to the equivariant derived category $D^+_{\sfrac{\Z}{N\Z}}(pt)$. This destroys the existence of contact squeezing and proves the main theorem (\textbf{Theorem \ref{main}}).

Throughout this paper we are considering not only bounded derived categories of sheaves but also locally bounded ones. By locally bounded sheaf $\F$ we mean for any relative compact open subset $V$, the pull-back sheaf $\F|_V$ is cohomologically bounded and we are working in this full subcategory. The notion of micro-support is adopted. Also there is no \textsl{a priori} notion of derived hom-functor for general triangulated categories, but in our paper all categories are obtained by localizations of derived categories of sheaves. So all $Rhom$-functors we implement here are inherited from the structure of cochain complexes. Due to the unboundedness issue it may be better to bear in mind the frameworks of  the homotopy category of cochain complexes, differential graded categories and model structures. All tensor products $\otimes$, as well as external products $\boxtimes$, stand for left derived functors of usual tensor products of sheaves. 

For more information about the interaction between sheaf theory and symplectic topology we would like to refer the reader to the excellent expository lecture by C. Viterbo \cite{Vi} which can be found on his website. It also contains many useful discussions about generating function apparatus and their application to symplectic geometry. It is Sandon who \cite{Sa} uses generating function techniques to define contact homology, by extending Traynor's symplectic homology \cite{Tr}, for open domains in $\R^{2n}\x\mathds{S}^1$ and gives a new proof of non-squeezability result of Eliashberg-Kim-Polterovich (\textbf{Theorem \ref{gap}}). It turns out that for open balls their symplectic/contact homology happen to coincide with our construction when forgetting $\sfrac{\Z}{N\Z}$-action. It will be very interesting if this coincidence holds for more general domains. 

Finally we need to mention that during the preparation of this paper, the author is told by Eliashberg that Fraser \cite{Fr} also proves the same non-squeezing theorem. In fact she has two proofs, one based on generating function techniques and another in the original context of \cite{EKP}. All known  approaches are using some sort of equivariant constructions inspired by Tamarkin several years ago. \vspace{0.5cm}

\subsection{Acknowledgement}\ \vspace{0.5cm}

I was told about this project by my thesis advisor, Dima Tamarkin and would like to thank deeply for his guidance and generosity. It is Dima who shares with me many profound ideas towards the result and useful techniques in microlocal sheaf theory. Also I need to express my gratitude to Nicolas Vichery and St\'{e}phane Guillermou for carefully reading the draft and pointing out several mistakes. I also benefit a lot from series of the work of Margherita Sandon. I would like to thank Eric Rubiel Dolores Cuenca and Ann Rebecca Wei for considerable discussions during my graduate studies. I am also grateful to Leonid Polterovich, Yasha Eliashberg and Maia Fraser for their helpful comments and suggestions to a list of essentially related works in this beautiful area. I would also like to thank Dusa McDuff, Andrea D'Agnolo, Pierre Schapira and Claude Viterbo for their encouragements. This paper will form my Ph.D. thesis at Northwestern University.\\

\section{The Contact Manifold $\R^{2n}\x \mathds{S}^1$ }\label{manifold}

\vspace{0.6cm}

\subsection{Contact Structure as Conic Symplectic Structure}\ \\

Let $\R^{2n}\x\R=\R^n\x\R^n\x\R=\{(\mathbf{q,p},z)|\mathbf{q,p}\in\R^n,z\in\R\} $ be a contact manifold with the standard contact form $\alpha=\mathbf{p}d\mathbf{q}+dz$. Note that the choice of 1-form $\alpha$ in this paper is different from $dz+\f{1}{2}(\q d\mathbf{p}-\mathbf{p} d\q)$ in \cite{EKP}. These two 1-forms are contactomorphic to each other. The advantage of choosing $\alpha$ is to turn  
the contact manifold $\R^{2n}\x\R$ into the $\R_{>0}$-quotient of its symplectization $Y$ where
$$
Y=T^*\R^n\x T^*_{>0}\R=\{(\mathbf{q,p},z,\zeta)|(\mathbf{q,p})\in T^*\R^n,(z,\zeta)\in T^*\R,\zeta>0   \}
$$
and the natural $\R_{>0}$-action given by $F_\lambda(\mathbf{q,p},z,\zeta)=(\mathbf{q},\lambda \mathbf{p},z,\lambda\zeta)$.

The manifold $Y$ is a symplectic submanifold of the cotangent bundle with the standard symplectic form $\omega=d\mathbf{p}\wedge d\mathbf{q}+d\zeta\wedge dz$, and it is easy to see that $\omega$ satisfies the relation $F_\lambda ^*\omega=\lambda\omega$ which gives rise to the conic symplectic structure of $Y$. Moreover, the infinitesimal $\R_{>0}$-action is given by the radial vector field $\p_\lambda=\mathbf{p}\p_{\mathbf{p}}+\zeta\p_\zeta$, and we can see that the 1-form $\alpha':=\iota_{\p_\lambda}\omega=\mathbf{p}d\mathbf{q}+\zeta dz$ is the pull-back of the standard contact form $\alpha$ of $\R^{2n}\x\R$ via the quotient map $Y\rightarrow\R^{2n}\x\R$ where $(\mathbf{q,p},z,\zeta)$ is sent to $(\mathbf{q,p}/\zeta,z)$.\\

\subsection{Lifting Contact Isotopy}\ \\

\begin{defn}
A map $\Phi:\R^{2n}\x \mathds{S}^1\rightarrow\R^{2n}\x \mathds{S}^1$ is called a Hamiltonian contactomorphism if there exists a contact isotopy $\{\Phi_s\}_{s\in I}: I\x\R^{2n}\x \mathds{S}^1\rightarrow\R^{2n}\x \mathds{S}^1$  such that each $\Phi_s$ is a contactomorphism and $\Phi_0=Id$, $\Phi_1=\Phi$. Here $I$ is an interval containing $[0,1]$. Furthermore, if there is a compact set $K$ such that each $\Phi_s$ is identity outside $K$, we call $\Phi$ is a Hamiltonian contactomorphism with compact support.
\end{defn}

Since the circle $\mathds{S}^1$ is the $\Z$-quotient of the real line $\R$ by the shifting action $\mathbf{T}: z\mapsto z+1$, this isotopy can be lifted to $\{\Phi_s\}_{s\in I} : I\x\R^{2n}\x\R\rightarrow\R^{2n}\x\R$ satisfying the following $\Z$-equivarient condition: $\Phi_s\circ\mathbf{T}=\mathbf{T}\circ\Phi_s$. 

  We can then lift $\{\Phi_s\}_{s\in I}$ again to obtain the conic  Hamiltonian symplectic isotopy $\{\Phi_s\}_{s\in I}:I\x Y\rightarrow Y$, satisfying $\R_{>0}$-homogeneity condition $\Phi_s\circ F_\lambda=F_\lambda\circ\Phi_s$ and $\Z$-equivarient condition $\Phi_s\circ\mathbf{T}=\mathbf{T}\circ\Phi_s$. By setting $X=\R^n\x\R=\{(\mathbf{q},z)| \mathbf{q}\in\R^n,z\in\R\}$ we have a conic $\Z$-equivarient symplectic isotopy $\Phi_s : Y=T^*_{\zeta>0}X\rightarrow T^*_{\zeta>0}X=Y$ for $s\in I$.
  
  Let $Y^\mathbf{a}$ denote the antipode symplectic manifold $(T^*_{\zeta<0}X,\omega)$ and $Y^\mathbf{a}\x Y$  the symplectic manifold equipped with symplectic form $\omega\oplus\omega$. For each individual $s$, we can construct a Lagragian submanifold $\Lambda_s$ by the graph of $\Phi_s$:
$$
\Lambda_s=\{(-y,\Phi_s(y))|y\in Y\}\subset Y^\mathbf{a}\x Y.
$$

The symplectic manifold  $Y^\mathbf{a}\x Y$ itself is the cotangent bundle $T^*_{\zeta_1<0,\zeta_2>0}(X\x X)$. The \textsl{sheaf quantization} of symplectomorphism $\Phi_s$ is by definition an object $\mathscr{S}_s$ of $D(X\x X)$ such that $SS(\mathscr{S}_s)\subset\Lambda_s$. 

When it comes to the Hamiltonian isotopy $\{\Phi_s\}_{s\in I}$, it will be more natural to keep track of the whole family $\{\Lambda_s\}_{s\in I}$ consisted of all Lagrangian graphs as $s$ varies throughout $I$. In fact we can lift this family to a single Lagrangian submanifold of the symplectic manifold $T^*I\x Y^\mathbf{a}\x Y$. Namely, let $h_s$ be the $s$-dependent Hamiltonian function associated to $\Phi_s$, consider the overall graph 
\begin{equation}\label{overall}
\Lambda = \{(s,-h_s(\Phi_s(y)),y,\Phi_s(y)) | s\in I,y\in Y\} \subset T^*I\x Y^\mathbf{a}\x Y.
\end{equation}

This overall graph $\Lambda$ is a conic Lagrangian submanifold of  $T^*I \x Y^\mathbf{a}\x Y\subset T^*(I\x X\x X)$. Moreover, any inclusion $\{s\}\x X\x X\hookrightarrow I\x X\x X$ is non-characteristic for $\Lambda$, so that any individual graph $\Lambda_s$ of $\Phi_s$ can be reconstructed from $\Lambda$ via symplectic reduction with respect to $T^*_s I$. In \textbf{Section \ref{invariant}} we will state a key theorem (\textbf{Theorem} \ref{quantization}) established in \texttt{\cite[Thm.16.3]{Gu}} that there exists  an essentially unique object $\mathscr{S}\in D(I\x X\x X)$ whose micro-support $SS(\mathscr{S})$ coincides the overall Lagrangian graph $\Lambda$ outside the zero-section and its pull-back to the initial time is isomorphic to the constant sheaf supported on the diagonal of $X$. This kind of sheaf $\mathscr{S}$ is called the \textsl{sheaf quantization} of $\{\Phi_s\}$ in \texttt{\cite{Gu}\cite{GuKaSch}}, in the sense that composing objects with kernel $\mathscr{S}$ gives rise to a family of  endofunctors on the category of sheaves that transform their micro-supports along with the given Hamiltonian isotopy $\{\Phi_s\}$. In fact, by further kernel composition with the dual of $\mathscr{S}$ we can construct sheaf theoretic contact isotopy invariants which take the shape of micro-supports into account (see \textbf{Proposition \ref{conj}}).\\

\subsection{Triangulated Categories Associated to Contact Balls}\label{tricat}\ \\

For any positive number $R$, denote the open ball of radius $R$ in $\R^{2n}$ by $B_R=\{(\mathbf{q,p})|\mathbf{q}^2+\mathbf{p}^2<R^2\}$. Let $U_R=B_R\x \mathds{S}^1$ be the contact ball in $\R^{2n}\x\R$. We can lift the contact ball $U_R$ to $\widetilde{U_R}=B_R\x\R$ in the covering contact manifold $(\R^{2n}\x\R,\alpha)$. Define the \textsl{conification} of $\widetilde{U_R}$ in $T^*_{\zeta>0}X$ to be the inverse image of the homogeneous quotient map of $\widetilde{U_R}$,
$$
C_R:=Cone(\widetilde{U_R})=\{(\mathbf{q,p},z,\zeta)|\zeta>0,(\mathbf{q,p}/\zeta,z)\in \widetilde{U_R} \}\subset T^*_{\zeta>0}X.
$$

We want to describe the category of objects micro-supported on this ball in the following sense:  Let $E:=\{\mathbf{q}\in\R^{n}\}$ and $X=E\x\R$. Fix a ground field $\K$, let $D(X)=D(E\x\R)$ denote the derived category of sheaves of vector spaces over $\K$ on $X$. Let $D_{\leq 0}(X)$ be the full subcategory of $D(X)$, consisting of all sheaves whose microlocal singular support is within $\{\zeta\leq 0\}$. We define a triangulated category $\D(X)$ to be the quotient \cite{Ta}
$$
\D(X)=\D(E\x\R):=D(X)/{D_{\leq 0}(X)}.
$$ 

  Let $D_{Y\setminus C_R}(X)\subset \D(X)$ be the full subcategory of all objects micro-supported on the conic closed set $Y\setminus  C_R$. We will see that the embedding functor $D_{Y\setminus C_R}(X)\hookrightarrow\D(X)$ has a left adjoint functor, hence it makes sense to write
    
\begin{defn}\label{DC}  
Let $\mathcal{D}_{C_R}(X):=\D(X)/D_{Y\setminus C_R}(X)$ be the left semi-orthogonal complement of $D_{Y\setminus C_R}(X)$.
\end{defn}  
  
  In \textbf{Section \ref{contactsupp}} we will see moreover this left adjoint functor $\D(X)\rightarrow D_{Y\setminus C_R}(X)$ is represented by a composition kernel $\mathscr{Q}_R\in D(X\x X)$. Similarly, the projection functor $\D(X)\rightarrow\mathcal{D}_{C_R}(X)$ is given by another composition kernel $\PR\in D(X\x X)$. Let $\widetilde{\Delta}=\{\q_1=\q_2,z_1\geq z_2\}$ be the subset of $X\x X$, then the constant sheaf $\K_{\widetilde{\Delta}}$ becomes the composition kernel of the identity endofunctor of $\D(X)$.
  
  The semi-orthogonal decomposition of $\D(X)$ is given by the fact that these three kernels fit into the following distinguished triangle in $D(X\x X)$:
 $$
 \PR\rightarrow\K_{\widetilde{\Delta}}\rightarrow\mathscr{Q}_R\xrightarrow{+1}  .
 $$

\vspace{1cm}

\section{Construction of the Projector}\label{projector} 
\vspace{1cm}

\subsection{Basics of Convolution and Composition}\ \\

Let $V$ be a vector space and $M_1,M_2,M_3$ be manifolds. We now define a \textsl{convolution functor}:
$$
D(M_1\x M_2\x V)\x D(M_2\x M_3\x V)\xrightarrow{\bu}D(M_1\x M_3\x V)
$$
as follows. Let
$$
p_{ij}:M_1\x M_2\x M_3\x V\x V\rightarrow M_i\x M_j\x V
$$
be the maps given by $p_{12}(m_1,m_2,m_3,v_1,v_2)=(m_1,m_2,v_1)$ , $p_{23}(m_1,m_2,m_3,v_1,v_2)=(m_2,m_3,v_2)$, and $p_{13}(m_1,m_2,m_3,v_1,v_2)=(m_1,m_3,v_1+v_2)$. Notice that $p_{13}$ takes sum of the vector part. 

\begin{defn}\label{convolution}

For objects $\F\in D(M_1\x M_2\x V)$ and $\G\in D(M_2\x M_3\x V)$ we define their convolution product by 
$$
\F\U{M_2}{\bu}\G:= Rp_{13!}(p_{12}^{-1}\F\otimes p_{23}^{-1}\G)\in D(M_1\x M_3\x V).
$$

\end{defn}
\vspace{0.2in}

For objects $\F\in D(V)$ and $\G\in D(M\x V)$, we identify $V$ with $\mathbf{pt}\x\mathbf{pt}\x V$ and $M\x V$ with $\mathbf{pt}\x M\x V$ or $M\x\mathbf{pt}\x V$. Then we can define a commutative operation
$$
\F\U{V}{\ast} \G := \F\U{\mathbf{pt}}{\bu}\G\cong\G\U{\mathbf{pt}}{\bu}\F\in D(M\x V).
$$

The following proposition  allows us to work with  categories of type $\mathcal{D}$. This symbol stands for left semi-orthogonal pieces.
\begin{prop}\label{*}(Proposition 2.2 of {\cite{Ta}})
Take $V=\R$ and $\F\in D(M\x\R)$ and let $\K_{[0,\infty)}\rightarrow\K_0$ be the restriction morphism in $D(\R)$. Then $\F\in\D(M\x\R)$ if and only if the natural morphism $\F\ast\K_{[0,\infty)}\rightarrow\F\ast\K_0\simeq\F$ is an isomorphism.
\end{prop}
\begin{cor}
Let $\F\in\D(M_1\x M_2\x\R)$ and $\G\in D(M_2\x M_3\x\R)$. We then have $\F\U{M_2}{\bu}\G\in\D(M_1\x M_3\x\R)$.
\end{cor}

Now let us introduce the concept of \textsl{Fourier transform} :

\begin{defn}\label{Fourier}
Let $E$ be a real vector space of dimension $n$. The Fourier transform with respect to $E=\{\mathbf{q}\}$ is a functor $\D(E\x\R)\xrightarrow{\wedge}\D(E^\star\x\R)$ given by
$$
\F\overset{\wedge\,}{\mapsto}\widehat{\F}:=\F\U{\mathbf{q}}{\bu}\K_{\{t+\mathbf{p}\mathbf{q}\geq0\}}[n].
$$
\end{defn}

The Fourier transform functor has the following effect on the micro-supports. Let $E^\star\x E \xrightarrow{\wedge} E\x E^\star$ be given by $(\mathbf{p},\mathbf{q})\xrightarrow{\wedge}(-\mathbf{q},\mathbf{p})$.

\begin{thm}\label{FourierSS}(Theorem 3.6 of \cite{Ta})
$SS(\widehat{\F})=\widehat{SS(\F)}$ up to a subset of $\{k\leq0\}$.
\end{thm}

 On the other hand, we have another operation called \textsl{kernel composition}. In contrast to the presence of vector summation in convolution, the notion of composition only involves projection maps of product manifolds. Let $N_1,N_2,N_3$ be manifolds accompanied with the projection maps $\pi_{i,j}:N_1\x N_2\x N_3\rightarrow N_i\x N_j$ for $i<j$. We define the composition product by a functor
  $$
  D(N_1\x N_2)\x D(N_2\x N_3)\xrightarrow{\circ} D(N_1\x N_3).
  $$

\begin{defn}\label{composition}
Let $\F\in D(N_1\x N_2)$ and $\G\in D(N_2\x N_3)$ we set
$$
\F\U{N_2}{\circ}\G:=R\pi_{13!}(\pi_{12}^{-1}\F\otimes \pi_{23}^{-1}\G)\in D(N_1\x N_3).
$$
\end{defn}

Throughout the paper we will need to switch between the notions of convolution and composition, so it is worth saying a few more about their relation. Following \cite{GuSch}, assume that $\gamma$ is a convex closed proper cone in $V$ and we have a constant sheaf $\K_\gamma\in D(V)$. Let $\gamma_l=\{(v_1,v_2)|v_1-v_2\in\gamma\}$ and $\gamma_r=\{(v_1,v_2)|v_2-v_1\in\gamma\}$. We then have two constant sheaves $\K_{\gamma_l},\K_{\gamma_r}\in D(V_1\x V_2)$, one translates the convolution to left composition and another to right composition. 

\begin{prop}\label{conv-comp}
For any sheaf $\F\in D(M\x V)$, we have the following isomorphisms in $D(M\x V)$ : 
$\F\U{V}{\ast}\K_\gamma \cong \K_{\gamma_l}\U{V_2}{\circ}\F \cong\F\U{V_1}{\circ} \K_{\gamma_r}$.
\end{prop}

The above relation can be understood in the sense of isomorphisms between endofunctors of $ D(M\x V)$. In the following subsections, we will describe how to characterize semi-orthogonal projection functors between certain triangulated categories defined by microlocal conditions in terms of convolution and composition.\\

\subsection{Quantizating Hamiltonian Actions}\label{action}\ \\ 

Recall that $X=\{(\mathbf{q},z)\}=\R^{n}\x\R=E\x\R$. Before we construct the kernel of the projector $\PR$ to $\mathcal{D}_{C_R}(X)$ as an object of $D(X\x X)$, we should first consider its  symplectic counterpart for open ball $B_R=\{\mathbf{q}^2+\mathbf{p}^2<R^2\}$ in $T^*E=\{(\mathbf{q,p})|\mathbf{q}\in E, \mathbf{p}\in E^{\star} \}$. By introducing a new real variable $\R=\{t\}$ and its cotangent coordinates $T^*\R=\{(t,k)\}$, we can define the notion of \textsl{sectional micro-supports} as follows:

\begin{defn}\label{sectional}
For object $\F\in\D(E\x\R)$  we define its sectional microlocal singular support $\mu S(\F)$ by the following subset of $T^*E$:
$$
\mu S(\F):=\{(\mathbf{q,p})\in T^*E | \exists t : \,(\mathbf{q,p},t,1)\in SS(\F)\}.
$$
\end{defn}

In other words, $\mu S(\F)$ is the projection of the section of the conic closed set $SS(\F)$ at the hyperplane $\{k=1\}$ on $T^*E$. Now let $D_{T^*E\setminus B}(E\x\R)\subset\D(E\x\R)$ be the full subcategory of all objects $\F$ whose sectional micro-supports $\mu S(\F)$ sit outside of the open symplectic ball $B_R$. To be specific, for any $\F\in D_{T^*E\setminus B}(E\x\R)\subset\D(E\x\R)$, we should have $\mu S(\F)\subset\{\mathbf{(q,p)}\in T^*E|\q^2+\mathbf{p}^2\geq R^2\}$.

Let $D_{T^*E\setminus B}(E\x\R)\hookrightarrow\D(E\x\R)$ be the embedding functor. We show later in \textbf{Theorem \ref{decomposition}} that this embedding functor has a left adjoint functor $\D(E\x\R)\rightarrow D_{T^*E\setminus B}(E\x\R)$ given by a convolution kernel $\mathcal{Q}_R\in D(E_1\x E_2\x\R)$. Put $\Delta=\{(\mathbf{q}_1,\mathbf{q}_2,t)| \mathbf{q}_1=\mathbf{q}_2,t\geq0\}\subset E_1\x E_2\x\R$. The constant sheaf $\K_\Delta\in D(E_1\x E_2\x\R)$ stands for the convolution kernel of the identity endofunctor. Notice that we are using the additive structure of $\R$ in convolution operation. Now consider the distinguished triangle in $D(E_1\x E_2\x\R)$:
$$
\pr\rightarrow\K_\Delta\rightarrow\mathcal{Q}_R\xrightarrow{+1}.
$$
\indent Denote the left semi-orthogonal complement of $D_{T^*E\setminus B}(E\x\R)$ in $\D(E\x\R)$ by $\mathcal{D}_{B}(E\x\R)$. \textbf{Theorem \ref{decomposition}} tells  that $\mathcal{D}_{B}(E\x\R)\cong\D(E\x\R)/D_{T^*E\setminus B}(E\x\R)$. One can call $\mathcal{D}_{B}(E\x\R)$ the category of objects micro-supported on the symplectic ball $B_R$. It turns out that the semi-orthogonal projection functor $\D(E\x\R)\rightarrow \mathcal{D}_{B}(E\x\R)$ is given by convolution with $\pr$.\\

Let us start with geometric properties of the shape of the ball: rotational symmetry. Let $H(\mathbf{q,p})=\q^2 + \mathbf{p}^2$ be the distinguished Hamiltonian function of classical harmonic oscillator, then $B_R=\{(\mathbf{q,p})\in T^*E|H(\mathbf{q,p})<R^2\}$. Consider the corresponding Hamiltonian equations:
$$
\f{d\mathbf{q}}{da}=\f{\p H}{\p \mathbf{p}}=2\mathbf{p}\quad\ ,  \quad\ \f{d\mathbf{p}}{da}=-\f{\p H}{\p \mathbf{q}}=-2\mathbf{q}.
$$
Its solution is given by Hamiltonian rotations $\mathds{H}:\R\rightarrow Sp(2n)$ such that for $a\in\R$,
$$
\left(\begin{array}{clr}
\mathbf{q}(a)\\
\mathbf{p}(a)
\end{array}\right)
=\mathds{H}(a)
\left(\begin{array}{clr}
\mathbf{q}(0)\\
\mathbf{p}(0)
\end{array}\right)=
\left(\begin{array}{clr}
\cos(2a) & \sin(2a)\\
-\sin(2a) & \cos(2a)
\end{array}\right)
\left(\begin{array}{clr}
\mathbf{q}(0)\\
\mathbf{p}(0)
\end{array}\right).
$$
\vspace{0.2cm}

We have a unique lifting  $\widetilde{\mathds{H}}:\R\rightarrow\widetilde{Sp}(2n)$ to the universal covering group of $Sp(2n)$ such that $\widetilde{\mathds{H}}(0)=e$. Denote the image of $\widetilde{\mathds{H}}$ by $G$. $G$ is isomorphic to $\R$. 

For time $a\in G$, put $(\mathbf{q}_1,\mathbf{p}_1)=(q(0),p(0))$ and $(\mathbf{q}_2,\mathbf{p}_2)=(\mathbf{q}(a),\mathbf{p}(a))$. Let us write down the change of variable formula given by the Hamiltonian rotation:  

\begin{equation}\label{subjectto}
\mathbf{q}_2=\f{\mathbf{q}_1+\sin(2a)\mathbf{p}_2}{\cos(2a)} \,,\, \mathbf{p}_1=\f{\sin(2a)\mathbf{q}_1+\mathbf{p}_2}{\cos(2a)}
\end{equation}

By \textsl{generating function} we mean a function $S(\mathbf{q}_1,\mathbf{q}_2)$ satisfying equations 
$$
\f{\p S}{\p \mathbf{q}_1}=-\mathbf{p}_1 \quad , \quad \f{\p S}{\p \mathbf{q}_2}=\mathbf{p}_2 .
$$

By combining them with the explicit expression of the rotation $\mathds{H}(a)$ the above equations allow a solution (locally depends on $a$) up to a constant:
\begin{equation}\label{GF}
S(\mathbf{q}_1,\mathbf{q}_2)=\f{\cos(2a)}{2\sin(2a)}(\mathbf{q}_1^2+\mathbf{q}_2^2)-\f{\mathbf{q}_1\mathbf{q}_2}{\sin(2a)}.
\end{equation}

Next, substitute the variable $\mathbf{q}_2$ by $\mathbf{q}_1\cos(2a)+\mathbf{p}_1\sin(2a)$ and write $\mathbf{q}_1=\mathbf{q}$ and $\mathbf{p}_1=\mathbf{p}$ in the function $S(\mathbf{q}_1,\mathbf{q}_2)$, we get the local expression $S(\mathbf{q,p})$ appearing in our decomposition. Indeed, the globally defined generating function can be interpreted by the following \textsl{action functional} (see \cite{Mc} )
$$
S=\int_\gamma \mathbf{p}d\mathbf{q}-Hda
$$ 
where the path $\gamma$ is the integral curve of Hamiltonian equations, parametrized by $a$,  starting from $(\mathbf{q}_1,\mathbf{p}_1)$ and ending at $(\mathbf{q}_2,\mathbf{p}_2)$. If the time interval is contained in $(-\f{\pi}{2},\f{\pi}{2})$ we have the unique integral curve from $a=0$. For time interval longer than this, we can divide it into subintervals so that each of them admits an unique integral curve, and take the sum of their integral values we obtain $S$. 

Let us write down the local form of $S$ in $\q,\mathbf{p}$:
$$
S_a\mathbf{(q,p)}=\f{\sin(4a)}{4}(\mathbf{p}^2-\mathbf{q}^2)-\sin^2(2a)\mathbf{pq}.
$$

The above variables are actually vectors in $\R^n$ and here by writing their products we mean the inner products, for example $\q_1\q_2:=\langle\q_1,\q_2\rangle$. Let $\alpha:=\mathbf{p}d\mathbf{q}$ be the Liouville 1-form on $T^*E$ and put $\eta=Hda$. Let $\mathcal{A}:G\x T^*E_1\rightarrow T^*E_2$ be the corresponding Hamiltonian action via $\widetilde{\mathds{H}}$. Taking differential of the integral interpretation $S=\int_\gamma \mathbf{p}d\mathbf{q}-Hda$ gives us the decomposition
$$
dS=\mathcal{A}^*\alpha-\mathbf{p}d\q-Hda=\mathcal{A}^*\alpha-\alpha-\eta .
$$

Then we have the following proposition: 

\begin{prop}
The function $G\x T^*E\xrightarrow{S}\R$ leads to the decomposition 
$$
\mathcal{A}^*\alpha=\alpha+\eta+dS.
$$ 
\end{prop}
It is also easy to verify that $\mathcal{A}^*\alpha-\alpha=(\mathbf{q}^2+\mathbf{p}^2)da+dS$ explicitly. \\

With functions $H:G\x T^*E\rightarrow T^*G$ and $S:G\x T^*E\rightarrow\R$ we can perform the following  Legendrian embedding:
\begin{equation}\label{actiongraph}
\mathds{A}:G\x T^*E\rightarrow T^*G\x T^*E_1\x T^*E_2\x\R
\end{equation}
$$
\qquad\qquad\quad(a,x)\mapsto (-\eta(a,x),x^\mathbf{a},\mathcal{A}_a(x),-S_a(x))
$$
here $x:=(\mathbf{q}_1,\mathbf{p}_1)$ and its \textsl{antipode} $x^\mathbf{a}:=(\q_1,-\mathbf{p}_1)$ are in the cotangent space $ T^*E_1$. The minus sign in 1-form $-\eta(a,x)$ stands for the element $(a,-H(x))$ of the cotangent space $T^*G$. From the above decomposition it is easy to see that the image of $\mathds{A}$ (still denoted by $\mathds{A}$) is a Legendrian submanifold in $T^*G\x T^*E_1\x T^*E_2\x\R_t$. Let $L$ be the conification of this Legendrian in $T^*G\x T^*E_1\x T^*E_2\x T^*\R_t$ by
$$
L=Cone(\mathds{A}):=\{(k\xi,t,k)|(\xi,t)\in\mathds{A}, k>0\}
$$
then $L$ becomes a conic Lagrangian submanifold in $T^*G\x T^*E_1\x T^*E_2\x T^*\R_t\cong T^*(G\x E_1\x E_2\x\R_t)$. 

\begin{prop}\label{rotation}
There exists an object $\s$ of $D(G\x E_1\x E_2\x\R)$ such that $\s|_{0\x E_1\x E_2\x\R}\cong\K_\Delta$ where $\Delta=\{(\mathbf{q}_1,\mathbf{q}_2,t)|\mathbf{q}_1=\mathbf{q}_2, t\geq 0\}$ and $SS(\s)\cap T^*_{k>0}(G\x E_1\x E_2\x\R) \subset L$. The object $\s$ is called sheaf quantization of Hamiltonian rotations.
\end{prop}

\begin{proof}

In fact, the general theory of the existence and the uniqueness of sheaf quantization of Hamiltonian symplectomorphisms has been well-established by  \textbf{Proposition 3.2} and \textbf{Lemma 3.3} in the paper \cite{GuKaSch}. What we are going to do is to look for an incarnation of $\s$ for the sake of  the computation of the contact invariant in the present paper.

For $0<a<\f{\pi}{2}$, let us first construct a constant sheaf in $\D((0,\f{\pi}{2})\x E_1\x E_2\x\R)$ by
$$
\s :=\K_{\{(a,\mathbf{q}_1,\mathbf{q}_2,t)|t+S_a(\mathbf{q}_1,\mathbf{q}_2)\geq 0 \}}.
$$

It is easy to see that we have
$$
SS(\s)=\{(kdS,k)|k>0 ,t=-S_a(\mathbf{q}_1,\mathbf{q}_2)\}\bigcup \{k=0\}.
$$

According to (\ref{actiongraph}) one has $SS(\s)\cap T^*_{k>0}\subset L$ .

Note that the function 
$$
S_a(\mathbf{q}_1,\mathbf{q}_2)=\f{\cos(2a)}{2\sin(2a)}(\mathbf{q}_1^2+\mathbf{q}_2^2)-\f{\mathbf{q}_1\mathbf{q}_2}{\sin(2a)}
$$ 
is undefined at $a=0$ and $a=\f{\pi}{2}$. To see the behavior as $a$ approaches to $0$ from the right, let us rewrite the generating function as follows:
$$
S_a(\mathbf{q}_1,\mathbf{q}_2)=\f{\tan(a)}{2}(\mathbf{q}_1^2+\mathbf{q}_2^2)+\f{1}{2\sin(2a)}(\mathbf{q}_1-\mathbf{q}_2)^2.
$$

For those $\mathbf{q}_1-\mathbf{q}_2\neq0$, the function $S_a$ goes to the positive infinity as $a\rightarrow0^+$. Hence the sheaf $\s$ has no support for $\R_t$. On the other hand, for those $\mathbf{q}_1-\mathbf{q}_2=0$, $S_a$ goes to zero when $a\rightarrow0^+$. In this case the support for $\R_t$-coordinate is $t\geq0$. Thus by joining $a=0$ to the domain of $\s$ we can define 
\begin{equation}\label{leftend}
\s|_{[0,\f{\pi}{2})\x E_1\x E_2\x\R}\cong\K_{\{0<a<\f{\pi}{2},t+S_a(\mathbf{q}_1,\mathbf{q}_2)\geq0\}\bigcup\{(0,\mathbf{q},\mathbf{q},t)|t\geq0\}}.
\end{equation}

Let $\Delta=\{\q_1=\q_2,t\geq0\}\subset E_1\x E_2 \x\R$, we see that $\s|_{a=0}\cong\K_{\Delta}$ is the convolution kernel of the identity functor of $\D(E\x\R)$.

On the other hand, to see the behavior of $\s$ when $a$ approaches to $\pi/2$, let us rewrite the generating function as the following form:
$$
S_a(\mathbf{q}_1,\mathbf{q}_2)=\f{\cos(2a)}{2\sin(2a)}(\mathbf{q}_1^2+\mathbf{q}_2^2) - \f{1}{\sin(2a)}\mathbf{q}_1\mathbf{q}_2=\f{\mathbf{q}_1^2+\mathbf{q}_2^2}{2\tan(a)}-\f{(\mathbf{q}_1+\mathbf{q}_2)^2}{2\sin(2a)}.
$$

For those $\mathbf{q}_1+\mathbf{q}_2\neq0$, when $a$ approaches to $\f{\pi}{2}$ from the left, the function $S_a(\mathbf{q}_1,\mathbf{q}_2)$ goes to the negative infinity, hence the sheaf is supported on the whole $\R_t$. And for $\mathbf{q}_1+\mathbf{q}_2=0$, $S_a(\mathbf{q}_1,\mathbf{q}_2)$ goes to zero when $a$ goes to $\f{\pi}{2}$.  Thus by joining $a=\f{\pi}{2}$ we define
\begin{equation}\label{rightend}
\s|_{(0,\pi/2]\x E_1\x E_2\x\R}\cong\K_{\{0<a<\f{\pi}{2} \, , \,t+S_a(\mathbf{q}_1,\mathbf{q}_2)\geq0\}\bigcup\{(\f{\pi}{2} ,\mathbf{q,-q},t)|t\geq0\}\bigcup\{a=\f{\pi}{2} ,\mathbf{q}_1+\mathbf{q}_2\neq 0\}}.
\end{equation}

This way one can even extend $\s$ to an object of $\D((-\f{\pi}{2},\f{\pi}{2}]\x E_1\x E_2\x\R)$. However, due to the periodicity of the function $S_a$ this construction cannot be extended naively outside the interval $(-\f{\pi}{2},\f{\pi}{2}]$ since there will be multiple Hamiltonian flow trajectories from $\mathbf{q}_1$ to $\mathbf{q}_2$. Instead, having all data lifted to the universal covering $G$, one convolutes local quantizations and constructs the global quantization $\s$.

Note that we need $\s$ to act on the category $\D(E\x\R)$ by convolution, so one expects  $\s|_{a_1+a_2}\cong \s|_{a_1}\U{E}{\bu} \s|_{a_2}$ for all $a_1,a_2$ in $G$. Indeed, let $a_1,a_2$ in $ [0,\f{\pi}{2})$ and write down their local expressions $\s|_{a_1}=\K_{\{(\q_1,\q_2,t)|t+S_{a_1}(\q_1,\q_2)\geq0\}}\in\D(E_1\x E_2\x\R)$ and  $\s|_{a_2}=\K_{\{(\q_2,\q_3,t)|t+S_{a_2}(\q_2,\q_3)\geq0\}}\in\D(E_2\x E_3\x\R)$. 
  Let $\rho$ be the projection along $E_2$. We then define stalkwisely
\begin{equation}\label{explicit form} 
\s|_{a_1+a_2}= R\rho_!\K_{\{(\q_1,\q_2,\q_3,t)|t+S_{a_1}(\q_1,\q_2)+S_{a_2}(\q_2,\q_3)\geq0\}}.
\end{equation}

One extends this out by defining $\s|_{a_1+a_2+\cdots+a_m}= \s|_{a_1}\U{E}{\bu} \s|_{a_2}\U{E}{\bu}\cdots\U{E}{\bu}\s|_{a_m}$. The extension to $\s|_{\{a\leq0\}}$ is similar. By the integral form of the action functional this extension $\s$ only depends on the sum $\Sigma a_j$ and it is easy to see $SS(\s)\cap T^*_{k>0}\subset L$.

\end{proof}

\vspace{1cm}
\subsection{Objects Micro-Supported on a Symplectic Ball}\ \\

Once we obtain the sheaf quantization $\s\in\D(G\x E_1\x E_2\x\R)$, we can perform Fourier transform (\textbf{Definition \ref{Fourier}}) on it with respect to $G$ ( i.e. $a$-coordinate). Set 
$$
\widehat{\s}:=\s\U{a}{\bu}\K_{\{(a,b,t)|t+ab\geq 0\}}[1]\in\D(G^\star\x E_1\x E_2\x\R).
$$
Since $SS(\s)\subset L$, we have $SS(\widehat{\s})\subset\widehat{L}$ where the hat $\wedge:T^*G\rightarrow T^*G^\star$ denotes the image under $(a,b)\mapsto(-b,a)$. Observe that for $Cone(\mathds{A})=L\subset T^*(G\x E_1\x E_2\x\R)$ the map $\wedge$ operates on $\{(a,-H(x))\}\subset T^*G$ whenever there is a point $(a,x)\in G\x T^*E$ whose value $S(a,x)$ hits the last coordinate of $\mathds{A}$. By (\ref{actiongraph}) and \textbf{Theorem \ref{FourierSS}} we get a useful estimation of our sectional micro-support:
 \begin{equation}\label{SSS}
\mu S(\widehat{\s})\subset\{(b,a,\mathbf{q}_1,\mathbf{p}_1,\mathbf{q}_2,\mathbf{p}_2)|(\mathbf{q}_2,\mathbf{p}_2)=\mathcal{A}_a(\mathbf{q}_1,\mathbf{p}_1),\; b=H(\mathbf{q}_1,\mathbf{p}_1)=H(\mathbf{q}_2,\mathbf{p}_2)\}.
\end{equation}

The object $\widehat{\s}$ will be our main ingredient to construct the desired projector. First, convolute $\widehat{\s}$ with the constant sheaf $\K_\R$:
\begin{equation}\label{Delta}
\widehat{\s}\U{b}{\circ}\K_\R=(\s\U{a}{\bu}\K_{\{t+ab\geq0\}}[1])\U{b}{\circ}\K_\R
\cong\s\U{a}{\bu}(\K_{\{t+ab\geq0\}}\U{b}{\circ}\K_\R)[1]
\end{equation}
$$
\cong\s\U{a}{\bu}\K_{\{a=0,t\geq 0\}}\cong \s|_{a=0}\cong\K_\Delta.
$$

This gives us the kernel representing the identity endofunctor on $\D(E\x\R)$. Now let $\K_{\{b<R^2\}}\in D(G^\star)$ be the constant sheaf supported on the open set $\{b<R^2\}$ and define 
\begin{equation}\label{P}
\pr := \widehat{\s}\U{b}{\circ}\K_{\{b<R^2\}}\in\D(E_1\x E_2\x\R)
\end{equation}
and its counterpart
\begin{equation}\label{Q}
\mathcal{Q}_{R} := \widehat{\s}\U{b}{\circ}\K_{\{b\geq R^2\}}\in\D(E_1\x E_2\x\R)
\end{equation}
so that we can form the following distinguished triangle in $\D(E_1\x E_2\x\R)$:
$$
\pr\rightarrow\K_\Delta\rightarrow\mathcal{Q}_R\xrightarrow{+1}.
$$

Recall that we define $D_{T^*E\setminus B}(E\x\R)$ by the full subcategory of all objects micro-supported in the closed set $T^*E\setminus B$ and ${\mathcal{D}}_B(E\x\R)$ to be its left semi-orthogonal complement. One might think that ${\mathcal{D}}_B(E\x\R) $ is some sort of subcategory of objects micro-supported on $B_R$. The precise statement is included in the following theorem:

\begin{thm}\label{decomposition}
Convolution with the distinguished triangle ($\pr\rightarrow\K_\Delta\rightarrow\mathcal{Q}_R\xrightarrow{+1}$) over $E$ gives a semi-orthogonal decomposition of the triangulated category $\D(E\x\R)$ with respect to the null system $D_{T^*E\setminus B}(E\x\R)$.
\end{thm}
\begin{proof}

There are three major steps in constructing a semi-orthogonal decomposition of a given triangulated category. First, we identify a null system $D_{T^*E\setminus B}(E\x\R)$ and its projector functor given by $\mathcal{Q}_R$. Second, we identify its left semi-orthogonal complement $\mathcal{D}_B(E\x\R)$ and the corresponding projector functor $\pr$. Finally we show that these two projectors, together with the identity endofunctor, give the desired decomposition of $\D(E\x\R)$. We refer the reader to \textbf{Section \ref{appendix} (Appendix)} for functorial properties of micro-supports which are used in the proof.\\

\textbf{Step 1}. Pick an arbitrary object $\G$ of $D(E_1\x\R)$, we are going to show that $\G\U{E_1}{\bu}\mathcal{Q}_R$ is an object of $D_{T^*E\setminus B}(E\x\R)$. Let us define the following notation of the maps:

\xymatrix{
&&& {G^\star}\x E_1\x E_2\x\R_1\x\R_2 \ar[ld]_{p_1} \ar[d]^{p} \ar[rd]^{q_2}   \\
&& E_1\x\R_1 & E_2\x\R &  {G^\star}\x E_1\x E_2\x\R_2   \\
}
\noindent here $p_1$ and $q_2$ are projections, and $p$ performs summation $\R_1\x\R_2\rightarrow \R$ and projection for the rest of components. By (\ref{Q}) we have
$$
\G\U{E_1}{\bu}\mathcal{Q}_R\cong \G\U{E_1}{\bu}\widehat{\s}\U{b}{\circ}\K_{b\geq R^2}
$$
$$
\cong Rp_!(p_1^{-1}\G\otimes q_2^{-1}\widehat{\s}\otimes \K_{[R^2,\infty)\x E_1\x E_2\x \R_1\x\R_2}).
$$

Let us estimate the corresponding micro-supports. By \textbf{Proposition \ref{pull}} and (\ref{SSS}) and \textbf{Theorem \ref{FourierSS}} we have
$$
SS(p_1^{-1}\G)\subset\{b,0,\mathbf{q}_1,k\mathbf{p}_1,\mathbf{q}_2,0,t_1,k,t_2,0|k\geq0  \}
$$
$$
SS(q_2^{-1}\widehat{\s})\subset\{(b,k'\beta,\mathbf{q}_1,k'\mathbf{p}_1',\mathbf{q}_2,k'\mathbf{p}_2',t_1,0,t_2,k') |k'\geq0,(\mathbf{q}_2,\mathbf{p}_2')=\mathcal{A}_\beta(\mathbf{q}_1,\mathbf{p}_1'), H(\mathbf{q}_i,\mathbf{p}_i')=b   \}.
$$
It is clear that 
$$
SS(p_1^{-1}\G) \cap SS(q_2^{-1}\widehat{\s})^\mathbf{a} \subset  T_{{G^\star}\x E_1\x E_2\x\R_1\x\R_2}^* ({G^\star}\x E_1\x E_2\x\R_1\x\R_2),
$$
thus by \textbf{Proposition \ref{tenhom}(i)} we have estimation
\begin{equation}\label{Mary}
SS(p_1^{-1}\G\otimes q_2^{-1}\widehat{\s})\subset SS(p_1^{-1}\G) +SS(q_2^{-1}\widehat{\s})
\end{equation}
$$
\subset \{(b,k'\beta,\mathbf{q}_1,k\mathbf{p}_1+k'\mathbf{p}_1',\mathbf{q}_2,k'\mathbf{p}_2',t_1,k,t_2,k'|k,k'\geq0,H(\mathbf{q}_i,\mathbf{p}_i')=b \}.
$$
On the other hand we know that 
\begin{equation}\label{Mandy}
SS(\K_{[R^2,\infty)\x E_1\x E_2\x \R_1\x\R_2})=\{(R^2,k'')|k''\geq0\}\x T_{ E_1\x E_2\x\R_1\x\R_2}^* ( E_1\x E_2\x\R_1\x\R_2)\}
\end{equation}
$$
\bigcup \{(b,0)|b>R^2 \}\x T_{ E_1\x E_2\x\R_1\x\R_2}^* ( E_1\x E_2\x\R_1\x\R_2)\}.
$$
It is also easy to see from (\ref{Mary}) and (\ref{Mandy}) we have
$$
SS(p_1^{-1}\G\otimes q_2^{-1}\widehat{\s})\cap SS(\K_{[R^2,\infty)\x E_1\x E_2\x \R_1\x\R_2})^\mathbf{a} \subset  T_{{G^\star}\x E_1\x E_2\x\R_1\x\R_2}^* ({G^\star}\x E_1\x E_2\x\R_1\x\R_2).
$$

Let us denote 
$$
\mathscr{H}:=p_1^{-1}\G\otimes q_2^{-1}\widehat{\s}\otimes \K_{[R^2,\infty)\x E_1\x E_2\x \R_1\x\R_2}.
$$

Apply \textbf{Proposition \ref{tenhom}(i)} on (\ref{Mary})(\ref{Mandy}) we get the following estimate
\begin{equation}\label{Macy}
SS(\mathscr{H})\subset SS(p_1^{-1}\G\otimes q_2^{-1}\widehat{\s})+SS(\K_{[R^2,\infty)\x E_1\x E_2\x \R_1\x\R_2})
\end{equation}
$$
\subset \{(R^2,k''+k'\beta,\mathbf{q}_1,k\mathbf{p}_1+k'\mathbf{p}_1',\mathbf{q}_2,k'\mathbf{p}_2',t_1,k,t_2,k')|k,k',k''\geq0 ,H(\mathbf{q}_i,\mathbf{p}_i')=R^2   \}
$$
$$
\bigcup\{(b,k'\beta,\mathbf{q}_1,k\mathbf{p}_1+k'\mathbf{p}_1',\mathbf{q}_2,k'\mathbf{p}_2',t_1,k,t_2,k'|k,k'\geq0,H(\mathbf{q}_i,\mathbf{p}_i')=b>R^2 \}.
$$

Before qualifying the members of the set $SS(Rp_!\mathscr{H})$, let us factor the map $p$ into a summation map $\sigma$ followed by a projection $\pi$:
$$
p: {G^\star}\x E_1\x E_2\x\R_1\x\R_2\xrightarrow{\sigma}{G^\star}\x E_1\x E_2\x\R_1\x\R_2\xrightarrow{\pi} E_2\x\R_2,
$$
where $\sigma : (t_1,t_2)\mapsto (t_1,t_1+t_2)$ and the rest of them are projections. Then apply \textbf{Proposition \ref{push}}  on (\ref{Macy}) we have 
\begin{equation}\label{Micky}
SS(R\sigma_!\mathscr{H})\subset \{(R^2,k''+k'\beta,\mathbf{q}_1,k\mathbf{p}_1+k'\mathbf{p}_1',\mathbf{q}_2,k'\mathbf{p}_2',t_1,k-k',t_2,k')|k,k',k''\geq0 ,H(\mathbf{q}_i,\mathbf{p}_i')=R^2   \}
\end{equation}
$$
\bigcup\{(b,k'\beta,\mathbf{q}_1,k\mathbf{p}_1+k'\mathbf{p}_1',\mathbf{q}_2,k'\mathbf{p}_2',t_1,k-k',t_2,k'|k,k'\geq0,H(\mathbf{q}_i,\mathbf{p}_i')=b>R^2 \}.
$$

Denote the above union of sets by $Z$, and define 
$$
\rho:T^*({G^\star}\x E_1\x E_2\x \R_1\x\R_2)\rightarrow T^*(E_2 \x\R_2)\x {G^\star}^\star\x  E_1^\star \x  \R_1^\star
$$ 
to be the projection along ${G^\star}\x E_1\x\R_1$, and let 
  $$
  i : T^*(E_2\x \R_2) \hookrightarrow  T^*( E_2\x \R_2)\x {G^\star}^\star \x E_1^\star \x  \R_1^\star
  $$
be the embedding into the zero of cotangent components ${G^\star}^\star\x E_1^\star \x  \R_1^\star
$.
For any vector $u=(\mathbf{q}_2,\mathbf{p},t_2,1)\in SS(R\pi_!(R\sigma_!\mathscr{H}))\subset T^*(E_2\x \R_2)$, apply \textbf{Proposition \ref{key}} on (\ref{Micky}) we have the following relation
$$
(-,0,-,0,\mathbf{q}_2,\mathbf{p},-,0,t_2,1)=i(u) \in \overline{\rho(SS(R\sigma_!\mathscr{H}))}\subset\overline{\rho(Z)}
$$
here $\overline{\rho(Z)}$ means the closure of $\rho(Z)$ and hyphens denote the arguments. Hence there are sequences $\{k\}$,$\{k'\}$,$\{\mathbf{p}_2'\}$ in terms of the corresponding coordinates of $Z$ satisfying the limit conditions $\{k'\}\rightarrow1$ and $\{k-k'\}\rightarrow0$ and $\{k'\mathbf{p}_2'\}\rightarrow \mathbf{p}$, respectively. This means that both $\{k\}$ and $\{k'\}$ approach to $1$, and then $\{\mathbf{p}_2'\}$ approaches to $\mathbf{p}$.

Since for any $\mathbf{p}_2'\in\{\mathbf{p}_2'\}$ we have $H(\mathbf{q}_2,\mathbf{p}_2')\geq R^2$, we must have its limit satisfying $H(\mathbf{q}_2,\mathbf{p})\geq R^2$. Hence the object $R\pi_!R\sigma_!\mathscr{H}$ must be a member of the category $D_{T^*E\setminus B}(E_2\x\R_2)$. So we finish the \textbf{Step 1} by concluding that
$$
\G\U{E_1}{\bu}\mathcal{Q}_R\cong Rp_!\mathscr{H}\cong R\pi_!R\sigma_!\mathscr{H} \in D_{T^*E\setminus B}(E_2\x\R_2).
$$
\textbf{Step 2.} We need to prove that convolution with $\pr$ projects onto the left semi-orthogonal complement of $D_{T^*E\setminus B}(E\x\R)$, which means for any object $\F$ in $D_{T^*E\setminus B}(E\x\R)$ and object $\G$ in $\D(E\x\R)$, we have $Rhom(\G \U {E_1}{\bu}\pr,\F)\cong 0$.

Let us define a bunch of maps in the following:  let 
$$
p_1:{G^\star}\x E_1\x E_2\x\R_1\x\R_2\xrightarrow{q_1}{G^\star}\x E_1\x\R_1\xrightarrow{\pi_1}E_1\x\R_1
$$
be successive projections onto the domain of $\G$, and 
$$
p: {G^\star}\x E_1\x E_2\x\R_1\x\R_2\xrightarrow{s}{G^\star}\x E_2\x\R\xrightarrow{\pi_2}E_2\x\R
$$
onto the domain of $\F$. Here the map $s$ performs summation $\R_1\x\R_2\rightarrow\R$, and for the rest of the components $s$ is projection. $\pi_2$ is projection along ${G^\star}$.

  Next, we want to describe the restriction to $(-\infty,R^2)$. Let 
\begin{equation*}
\begin{aligned}
(-\infty,R^2)\x E_1\x E_2\x\R_1\x\R_2\xrightarrow{j} {G^\star}\x E_1\x E_2\x\R_1\x\R_2  \\
(-\infty,R^2)\x E_1\x\R_1\xrightarrow{j_1} {G^\star}\x E_1\x\R_1\\
(-\infty,R^2)\x E_2\x\R\xrightarrow{j_2} {G^\star}\x E_2\x\R
\end{aligned}
\end{equation*}
be open embeddings of $(-\infty,R^2)$ into ${G^\star}$. Also define the projection map ${G^\star}\x E_1\x E_2\x\R_1\x\R_2\xrightarrow{q_2}{G^\star}\x E_1\x E_2\x\R_2$ onto the domain of $\widehat{\s}$. \\

\xymatrix{
&{G^\star}\x E_1\x E_2\x\R_2 & {G^\star}\x E_1\x E_2\x\R_1\x\R_2  \ar@/_/[lldd]_{q_1} \ar@/^3pc/[rdd]^{s} \ar@/^8pc/[rddd]^{p}   \ar[l]_{q_2}   & \\
& & (-\infty,R^2)\x E_1\x E_2\x\R_1\x\R_2 \ar@{^{(}->}[u]^{j}  \ar@{-->}[d]^{s} \ar@{-->}[dl]^{q_1}&   \\
{G^\star}\x E_1\x\R_1\ar[d]_{\pi_1}   & (-\infty,R^2)\x E_1\x\R_1\ar@{_{(}->}[l]_{j_1} &(-\infty,R^2)\x E_2\x\R\ar@{^{(}->}[r]^{j_2}  &    {G^\star}\x E_2\x\R  \ar[d]_{\pi_2} \\
E_1\x \R_1 & & &  E_2\x\R \\
&
}

With the above maps and (\ref{P}) we can unwrap $Rhom(\G \U {E_1}{\bu}\pr,\F)$ as the following:
$$
Rhom(\G \U {E_1}{\bu}\pr,\F)
$$
$$
\cong Rhom(\G \U {E_1}{\bu}(\widehat{\s}\U{b}{\circ}\K_{b<R^2})    ;\F)
$$
$$
\cong Rhom((\G \U {E_1}{\bu}\widehat{\s})\U{b}{\circ}\K_{b<R^2}    ;\F)
$$
$$
\cong Rhom(R\pi_{2!} Rj_{2!}j_2^{-1} (Rs_!(p_1^{-1}\G\otimes q_2^{-1}\widehat{\s})   )      ;\F)
$$
$$
\cong Rhom(R\pi_{2!} Rj_{2!}Rs_!  j^{-1} ( p_1^{-1}\G\otimes q_2^{-1}\widehat{\s} )        ;\F)
$$
$$
\cong Rhom(R\pi_{2!} Rs_! Rj_! j^{-1} ( p_1^{-1}\G\otimes q_2^{-1}\widehat{\s} )        ;\F)
$$
$$
\cong Rhom(Rp_! Rj_!j^{-1}(q_1^{-1}\pi_1^{-1}        \G\otimes q_2^{-1}\widehat{\s}  ) ;\F)
$$
$$
\cong Rhom(Rp_! ( (Rj_!j^{-1}q_1^{-1}\pi_1^{-1}        \G)\otimes q_2^{-1}\widehat{\s}  ) ;\F)
$$
$$
\cong Rhom(Rp_! ((  Rj_!q_1^{-1}j_1^{-1}\pi_1^{-1}        \G)\otimes q_2^{-1}\widehat{\s})   ;\F)
$$
$$
\cong Rhom ( Rj_!q_1^{-1}j_1^{-1}\pi_1^{-1}        \G  ; \underline{Rhom}(q_2^{-1}\widehat{\s} ;p^{!}\F  ))
$$
$$
\cong Rhom  (  j_1^{-1}\pi_1^{-1}        \G  ;Rq_{1*}j^!\underline{Rhom}(q_2^{-1}\widehat{\s} ;p^{!}\F  ))
$$
$$
=Rhom  (  j_1^{-1}\pi_1^{-1}        \G  ;Rq_{1*}j^{-1}\underline{Rhom}(q_2^{-1}\widehat{\s} ;p^{!}\F  )).
$$

Denote the object $\underline{Rhom}(q_2^{-1}\widehat{\s} ;p^{!}\F  )$ by $\mathscr{K}$ then the above sequence of adjunctions tells us 
\begin{equation}\label{Beth}
Rhom(\G \U {E_1}{\bu}\pr,\F)\cong Rhom  (  j_1^{-1}\pi_1^{-1}        \G  ;Rq_{1*}j^{-1} \mathscr{K}).
\end{equation}

Let us estimate the corresponding micro-supports with the help of (\ref{SSS}) \textbf{Proposition \ref{pull}} :
\begin{equation}\label{Bob}
SS(q_2^{-1}\widehat{\s})\subset\{(b,k\beta,\mathbf{q}_1,k\mathbf{p}_1,\mathbf{q}_2,k\mathbf{p}_2,t_1,0,t_2,k) |k\geq0,(\mathbf{q}_2,\mathbf{p}_2)=\mathcal{A}_\beta(\mathbf{q}_1,\mathbf{p}_1), H(\mathbf{q}_i,\mathbf{p}_i)=b   \}.
\end{equation}
Recall that from $\F\in D_{T^*E\setminus B}(E\x\R)$ we have
\begin{equation}\label{Ben}
SS((p^{!}\F  )\subset \{(b,0,\mathbf{q}_1,0,\mathbf{q}_2,k'\mathbf{p}_2',t_1,k',t_2,k')|k'>0,H(\mathbf{q}_2,\mathbf{p}_2')\geq R^2    \}\bigcup\{k'=0\}.
\end{equation}
From (\ref{Bob})(\ref{Ben}) is easy to see that 
$$
SS(q_2^{-1}\widehat{\s})\cap SS((p^{!}\F  )\subset T_{{G^\star}\x E_1\x E_2\x\R_1\x\R_2}^* ({G^\star}\x E_1\x E_2\x\R_1\x\R_2),
$$
so by \textbf{Proposition \ref{tenhom}(ii)} we have estimation
$$
SS( \underline{Rhom}(q_2^{-1}\widehat{\s} ,p^{!}\F  ) )\subset SS(q_2^{-1}\widehat{\s})^\mathbf{a} + SS((p^{!}\F  )
$$
$$
\subset \{ (b,-k\beta,\mathbf{q}_1,k\mathbf{p}_1,\mathbf{q}_2,k'\mathbf{p}_2'-k\mathbf{p}_2,t_1,k',k'-k)|k\geq0,k'>0,H(\mathbf{q}_i,\mathbf{p}_i)=b,H(\mathbf{q}_2,\mathbf{p}_2')\geq R^2   \}
$$
$$
\bigcup  T^*(G^\star\x E_1\x E_2) \x\R_1\x\R_2 \x\{  (0,-k)\in\R_1^{\star}\x\R_2^{\star}|k\geq0\}.
$$

Denote the above union of sets by $Z$ then the above inequality translates to
$$
SS(\mathscr{K})\subset Z.
$$

On the other hand, let us define
$$
\rho:T^*((-\infty,R^2)\x E_1\x E_2\x \R_1\x\R_2)\rightarrow T^*((-\infty,R^2)\x E_1 \x\R_1)  
$$ 
to be the projection along $E_2\x\R_2$, and let 
  $$
  i : T^*({G^\star}\x E_1\x \R_1) \hookrightarrow  T^*({G^\star}\x E_1\x \R_1)\x E_2^\star \x  \R_2^\star
  $$
  be the embedding to the zero section of the second cotangent component.

Now we assume that in $T^*((-\infty,R^2)\x E_1 \x\R_1)$ there exists a vector 
$$
v=(b,\beta,\mathbf{q}_1,\mathbf{p},t_1,1)\in SS(Rq_{1*}j^{-1}\mathscr{K}).
$$

By \textbf{Proposition \ref{key}}, the image $i(v)=(b,\beta,\mathbf{q}_1,\mathbf{p},-,0,t_1,1,-,0)$ must lie in the closure of $\rho(SS( j^{-1}\mathscr{K}  ,p^{!}\F  ) ))$ and hence in the closure of $\rho(Z)$. This means that there exist sequences $\{\mathbf{p}_2\},\{\mathbf{p}_2'\},\{k\},\{k'\}$ of cotangent coordinates of $Z$ such that $\{k'\mathbf{p}_2'-k\mathbf{p}_2\}\rightarrow 0$ and $\{k'\}\rightarrow 1$ and $\{k'-k\}\rightarrow 0$. Thus both $\{k'\}$ and $\{k\}$ approach to $1$ and since $\{\mathbf{p}_2\}$ ranges in the compact set $H(\mathbf{q}_2,\mathbf{p}_2)=b$ then both $\{\mathbf{p}_2\}$ and $\{\mathbf{p}_2'\}$ have the same finite limit. However, from $b\in(-\infty,R^2)$ we know that $H(\mathbf{q}_2,\mathbf{p}_2)=b < R^2=H(\mathbf{q}_2,\mathbf{p}_2')$, which leads to a contradiction. 

The above discussion shows that  
\begin{equation}\label{Bay}
Rq_{1*}j^{-1}\mathscr{K} \in D_{\leq0}( (-\infty,R^2)\x E_1 \x\R_1)).
\end{equation}

On the other hand since $\G\in \D( E_1\x\R_1)$ we have 
\begin{equation}\label{Belt}
j_1^{-1}\pi_1^{-1}   \G \in \D((-\infty,R^2)\x E_1\x\R_1).
\end{equation}

Then from (\ref{Beth})(\ref{Bay})(\ref{Belt}) we can conclude that 
$$
Rhom(\G \U {E_1}{\bu}\pr;\F)\cong Rhom  (  j_1^{-1}\pi_1^{-1}        \G  ;Rq_{1*}j^{-1}\mathscr{K})\cong 0.
$$

Hence $\pr$ is the convolution kernel of the projector onto the left semi-orthogonal complement of $D_{T^*E\setminus B}(E\x\R)$.\\

\textbf{Step 3}. We focus on the triangulated structure of $\D(E\x\R)$. Since $D_{T^*E\setminus B}(E\x\R)$ is defined by the full subcategory of objects whose sectional micro-supports are outside $B_R$, then for any object $\F$ it is clear that $\F\in D_{T^*E\setminus B}(E\x\R)$ if and only if $\F[1]\in D_{T^*E\setminus B}(E\x\R)$. Moreover, for any distinguished triangle in $\D(E\x\R)$:
$$
\F\rightarrow\G\rightarrow\mathscr{H}\xrightarrow{+1}
$$
we always have the estimation
$$
SS(\G)\subset SS(\F)\bigcup SS(\mathscr{H}).
$$

It implies that if both $\F$ and $\mathscr{H}$ live in $D_{T^*E\setminus B}(E\x\R)$, then so does $\G$. Also, $0\in D_{T^*E\setminus B}(E\x\R)$, hence $D_{T^*E\setminus B}(E\x\R)$ forms a null system in $\D(E\x\R)$. Furthermore, for any $\F$ and $\G$ in $\D(E\x\R)$, we have the relation
$$
SS(\F\bigoplus\G)=SS(\F) \bigcup SS(\G),
$$ 
thus the triangulated subcategory $D_{T^*E\setminus B}(E\x\R)$ is \textsl{thick}: condition $\F \bigoplus \G\in D_{T^*E\setminus B}(E\x\R)$ implies that both $\F$ and $\G$ live in $D_{T^*E\setminus B}(E\x\R)$. Let us denote its left semi-orthogonal complement by ${\mathcal{D}}_B(E\x\R)$. Now for any object $\F\in \D(E\x\R)$ we have a decomposition in terms of the following distinguished triangle:\\
$$
\F\U{E}{\bu}\pr\rightarrow \F \rightarrow \F\U{E}{\bu}\mathcal{Q}_R \xrightarrow{+1}
$$\\
where $ \F\U{E}{\bu}\pr\in {\mathcal{D}}_B(E\x\R)$ and $\F\U{E}{\bu}\mathcal{Q}_R\in D_{T^*E\setminus B}(E\x\R)$. This means that the embedding functor $D_{T^*E\setminus B}(E\x\R)\hookrightarrow \D(E\x\R)$ admits a right adjoint and there is an equivalence $D_{T^*E\setminus B}(E\x\R)\cong \D(E\x\R)/ {\mathcal{D}}_B(E\x\R)$. It is equivalent to say that ${\mathcal{D}}_B(E\x\R)\hookrightarrow \D(E\x\R)$ admits a left adjoint functor, and we have the following equivalence of triangulated categories:
$$
 {\mathcal{D}}_B(E\x\R) \cong \D(E\x\R)/ D_{T^*E\setminus B}(E\x\R). 
 $$

\end{proof}

\subsection{Interlude: the geometry of $\pr$}\ \\

 In this subsection we would like to mention the periodic structure of the projector $\pr$. First, let us unwrap $\pr$ in terms of the sheaf quantization $\s$ and its convolutions (or compositions). By (\ref{P}) we have
\begin{equation}\label{block}
\pr:=\widehat{\s}\U{b}{\circ}\K_{\{b<R^2\}}
\end{equation}
$$
=(\s\U{a}{\bu}\K_{\{t+ab\geq0\}}[1])\U{b}{\circ}\K_{\{b<R^2\}}
$$
$$
\cong\s\U{a}{\bu}(\K_{\{t+ab\geq0\}}\U{b}{\circ}\K_{\{b<R^2\}})[1]
$$
$$
\cong\s\U{a}{\bu}\K_{\{ (a,t)|a\leq 0 \, ,\, t+aR^2\geq 0\}}.
$$

Define $\mathfrak{G}:a\mapsto a-\f{\pi}{2}$ to be the translation along $G$ and $\mathfrak{T}:t\mapsto t-\f{\pi}{2} R^2$ the translation along $\R_t$. We have $\mathfrak{G}_*\s\cong \K_A[-n]\U{E}{\circ}\s$, here $A=\{(\mathbf{q,-q})\}$ is the anti diagonal in $E\x E$. Moreover, $\mathfrak{G}$ and $\mathfrak{T}$ interact in the following way:
$$
\pr\cong  \s\U{a}{\bu}\K_{\{ (a,t)|a\leq 0 \, ,\, t+aR^2\geq 0\}}
$$
$$
\cong \mathfrak{G}_*\s\U{a}{\bu}   \mathfrak{G}_*\K_{\{ (a,t)|a\leq 0 \, ,\, t+aR^2\geq 0\}}
$$
$$
\cong\K_A[-n] \U{E}{\circ}\s\U{a}{\bu}  \K_{\{a\leq\f{\pi}{2} \,,\,t+aR^2-\f{\pi}{2}R^2\geq0\}}
$$
$$
\cong \K_A[-n]\U{E}{\circ} \s \U{a}{\bu} \K_{\mathfrak{T}^{-1}\{ a\leq\f{\pi}{2} \, ,\, t+aR^2\geq 0 \}}
$$
$$
\cong (\K_A[-n] \U{E}{\circ}\s\U{a}{\bu} \mathfrak{T}^{-1}\K_{\{a\leq\f{\pi}{2}\,,\,t+aR^2\geq0\}})
$$
$$
\cong \mathfrak{T}^{-1}(\K_A[-n]\U{E}{\circ} \s \U{a}{\bu}\K_{\{a\leq\f{\pi}{2}\,,\,t+aR^2\geq0\}}).
$$

After applying $\mathfrak{T}$ on both sides and composing with $\K_A$ it becomes
$$
\K_A\U{E}{\circ}\mathfrak{T}(\pr)\cong\s\U{a}{\bu}\K_{\{a\leq\f{\pi}{2}\,,\,t+aR^2\geq0\}}[-n].
$$

On the other hand, from the closed embedding $\{a\leq0\}\hookrightarrow\{a\leq \f{\pi}{2}\}$ there is a restriction morphism 
$$
\gamma :\s\U{a}{\bu}\K_{\{a\leq\f{\pi}{2}\,,\,t+aR^2\geq0\}}{\rightarrow }\s\U{a}{\bu}\K_{\{a\leq0\,,\,t+aR^2\geq0\}}
$$
which is exactly the morphism
\begin{equation}\label{gam}
\gamma : \K_A[n]\U{E}{\circ}\mathfrak{T}(\pr)\rightarrow \pr.
\end{equation}

Let $\Gamma$ be the co-cone of $\gamma$. By co-cone we mean that $\Gamma$ can be embedded into the distinguished triangle
\begin{equation}\label{Gamma}
\Gamma\rightarrow \K_A[n]\U{E}{\circ}\mathfrak{T}(\pr)\rightarrow\pr\xrightarrow{+1},
\end{equation}
thus by (\ref{gam}) we have
\begin{equation}\label{gamS}
\Gamma\cong \s\U{a}{\bu}\K_{\{0<a\leq\f{\pi}{2}\,,\,t+aR^2\geq0\}}.
\end{equation}

The object $\Gamma$ is a constant sheaf concentrated in degree $0$ as follows. For $0<a<\f{\pi}{2}$, let
$$
f(a)=-S_a(\mathbf{q}_1,\mathbf{q}_2)-aR^2.
$$ 

The function $f(a)$ has two critical points(or one with multiplicity two) $a_1$ and $a_2$. We assume that $f(a_1)\geq f(a_2)$.  Let  $d=(\mathbf{q}_1\mathbf{q}_2)^2-R^2(\mathbf{q}_1^2+\mathbf{q}_2^2-R^2)$ and $D=\{(\mathbf{q}_1,\mathbf{q}_2)| d\geq0\}$.  

\begin{prop}
Let $\Sigma=\{(\mathbf{q}_1,\mathbf{q}_2,t)|(\mathbf{q}_1,\mathbf{q}_2)\in D\, ,\, f(a_2)\leq t<f(a_1)\}$  be the subset of $E_1\x E_2\x \R$.  We have $\Gamma\cong\K_\Sigma$.
\end{prop}
\begin{proof}
According to the construction of the sheaf quantization object $\s$(\textbf{Proposition \ref{rotation}}) we know from (\ref{rightend}) there is 
$$
\s|_{(0,\pi/2]\x E\x E\x\R}\cong\K_{\{0<a<\f{\pi}{2} \, , \,t+S_a(\mathbf{q}_1,\mathbf{q}_2)\geq0\}\bigcup\{(\f{\pi}{2} ,\mathbf{q,-q},t)|t\geq0\}\bigcup\{a=\f{\pi}{2} ,\mathbf{q}_1+\mathbf{q}_2\neq 0\}}.
$$

Now let $\pi_1:G\x E\x E\x \R_1\x\R_2\rightarrow G\x\R_1$ and $\pi_1:G\x E\x E\x \R_1\x\R_2\rightarrow G\x E\x E\x\R_2$ be the projection maps, and let $\pi:G\x E\x E\x \R_1\x\R_2\rightarrow G\x E\x E\x\R_t$ be the summation $t=t_1+t_2$. Also define $p:G\x E\x E\x\R_t\rightarrow E\x E\x\R_t$ to be the projection along $G$, then from (\ref{gamS}) we have
$$
\Gamma\cong \s\U{a}{\bu}\K_{\{0<a\leq\f{\pi}{2}\,,\,t+aR^2\geq0\}}
$$
$$
= Rp_!R\pi_!(\pi_1^{-1}\K_{\{ 0<a\leq\f{\pi}{2}\,,\,t+aR^2\geq0\}}\otimes\pi_2^{-1}\K_{\{0<a<\f{\pi}{2} \, , \,t+S_a(\mathbf{q}_1,\mathbf{q}_2)\geq0\}\bigcup\{(\f{\pi}{2},\mathbf{q,-q},t)|t\geq0\}\bigcup\{a=\f{\pi}{2} ,\mathbf{q}_1+\mathbf{q}_2\neq 0\}})
$$
$$
 \cong Rp_! \K_W,
$$
where the set $W\subset G\x E\x E\x\R_t$ is given by
$$
W=\{0<a<\f{\pi}{2},t+aR^2+S_a(\mathbf{q}_1,\mathbf{q}_2)\geq0\}\bigcup\{a=\f{\pi}{2},\mathbf{q}_1+\mathbf{q}_2=0,t+\f{\pi}{2}R^2\geq 0 \}.
$$

Note that the last term $\{a=\f{\pi}{2},\mathbf{q}_1+\mathbf{q}_2\neq0\}$ is vanished by taking compact support cohomology of the fiber of $\pi$.  Next, we need to examine the compactly supported cohomology of the fibers of $p$ for each $(\mathbf{q}_1,\mathbf{q}_2)$, and then give a geometric characterization of our co-cone $\Gamma$. \\

\textbf{Case 1. } For $\mathbf{q}_1+\mathbf{q}_2\neq 0$, we can consider the function $f(a)=-S_a(\mathbf{q}_1,\mathbf{q}_2)-aR^2$ on $0<a<\f{\pi}{2}$. The function $f$ goes to the negative infinity when $a$ goes to $\f{\pi}{2}$, and $f$ goes to the positive infinity when $a$ goes to $0$ while maintaining $\mathbf{q}_1\neq \mathbf{q}_2$. So it is clear that the object $Rp_!\K_W$, at any given $\mathbf{q}_1\pm \mathbf{q}_2\neq0$, is only possibly supported on those $t$ which sit between the critical values of $f(a)$. The corresponding critical points satisfy $\f{\p f}{\p a}=0$ and are exactly the solutions of the Hamilton-Jacobi equation for the function $H\mathbf{(q,p)}=\mathbf{q}^2+\mathbf{p}^2$ being fixed to $R^2$. It means that there exists $\mathbf{p}_1$ and $\mathbf{p}_2$ satisfying $H(\mathbf{q}_1,\mathbf{p}_1)=R^2$ and $(\mathbf{q}_2,\mathbf{p}_2)=\widetilde{\mathds{H}}(a)(\mathbf{q}_1,\mathbf{p}_1)$. From the equations (\ref{subjectto})
$$
\mathbf{p}_1=\f{\mathbf{q}_2-\mathbf{q}_1\cos(2a)}{\sin(2a)} \quad,\quad R^2=\mathbf{q}_1^2+\mathbf{p}_1^2
$$
we arrive to the quadratic equation in $\cos(2a)$:
$$
R^2 \cos^2(2a)-2\mathbf{q}_1\mathbf{q}_2 \cos(2a) + (\mathbf{q}_1^2+\mathbf{q}_2^2-R^2)=0.
$$

We see that for $|\xi|>1$, 
$$
R^2\xi^2-2\mathbf{q}_1\mathbf{q}_2\xi+(\mathbf{q}_1^2+\mathbf{q}_2^2-R^2)=(\xi\mathbf{q}_1-\mathbf{q}_2)^2+(R^2-\mathbf{q}_1^2)(\xi^2-1)
$$
is always a positive number. So whenever it possesses a real solution $\xi$, $\xi$ must be in $[-1,1]$, to which we can assign with a cosine value $\xi=\cos(2a)$. Let $d=(\mathbf{q}_1\mathbf{q}_2)^2-R^2(\mathbf{q}_1^2+\mathbf{q}_2^2-R^2)$ be its discriminant, then only when $d>0$ we have two distinct solutions  $\cos(2a_1)$ and $\cos(2a_2)$. Here we fix the order by setting $0\leq a_1\leq a_2\leq \f{\pi}{2}$. According to the boundary behavior of $S_a$ when $a$ approaches $0$ and $\f{\pi}{2}$ respectively, it is always true that $f(a_2)<f(a_1)$. Set $D=\{(\mathbf{q}_1,\mathbf{q}_2)| d\geq0\}$, so we know that in the case $(\mathbf{q}_1,\mathbf{q}_2)\in int(D)$, the object $Rp_!\K_W$ is supported by  
$$
\{(\mathbf{q}_1,\mathbf{q}_2,t)|(\mathbf{q}_1,\mathbf{q}_2)\in int(D) \,,\, f(a_2)\leq t<f(a_1)\}.
$$

\textbf{Case 2. }  For $\mathbf{q}_1=\mathbf{q}_2=\mathbf{q}$, the function $S_a(\mathbf{q}_1,\mathbf{q}_2)$ becomes $-\tan(a)\mathbf{q}^2$, and it is easy to see that $f(a)=\tan(a)\mathbf{q}^2-aR^2$ has only one critical point $a_2$ satisfying 
$$
\cos(a_2)=\f{|\mathbf{q}|}{R}.
$$
Since $f(a)$ goes to $0$ as $a$ goes to $0$, and $f(a)$ goes to the infinity as $a$ goes to $\f{\pi}{2}$, the object $Rp_!\K_W$ has support on 
$$
\{(\mathbf{q,q},t)|f(a_2)\leq t<0 \}.
$$

\textbf{Case 3. }   $\mathbf{q}_1=-\mathbf{q}_2:=\mathbf{q}$, the function $S_a(\mathbf{q}_1,\mathbf{q}_2)$ becomes $\f{\mathbf{q}^2}{\tan(a)}$. Let us consider the function $f(a)=-\f{\mathbf{q}^2}{\tan(a)}-aR^2$ where $a$ ranges over $(0,\f{\pi}{2}]$. Solving the critical point equation $\f{\p f}{\p a}=0$ we get the unique solution $a_1\in(0,\f{\pi}{2}]$ satisfying 
$$
\sin(a_1)=\f{|\mathbf{q}|}{R}.
$$

Since $f(a)$ goes to negative infinity as $a$ approaches to $0$ and goes to $-\f{\pi}{2}R^2$ as $a$ goes to $\f{\pi}{2}$, the object $Rp_!\K_W$ has support on the set
$$
\{(\mathbf{q,-q},t)|-\f{\pi}{2}R^2\leq t< f(a_1)\}.
$$

These three cases are depicted in the following graph. The horizontal X-axis represents $a$, and the vertical Y-axis parametrizes $t$.
\begin{figure}[ht] 
\begin{center}
\mbox{
\subfigure[Case 1]{\includegraphics[height= 4.4cm]{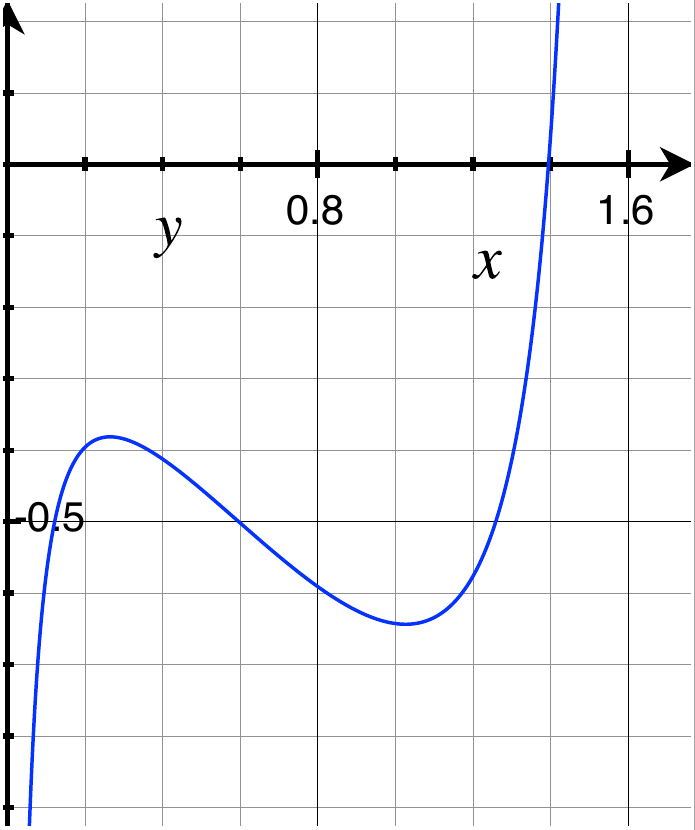}}\qquad
\subfigure[Case 2]{\includegraphics[height= 4.4cm]{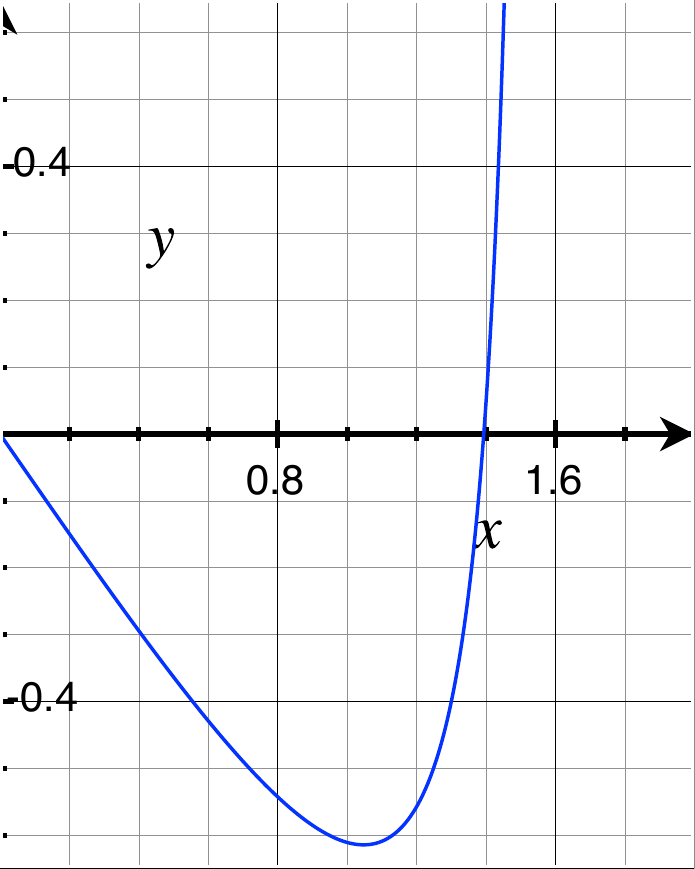}}\qquad
\subfigure[Case 3]{\includegraphics[height= 4.4cm]{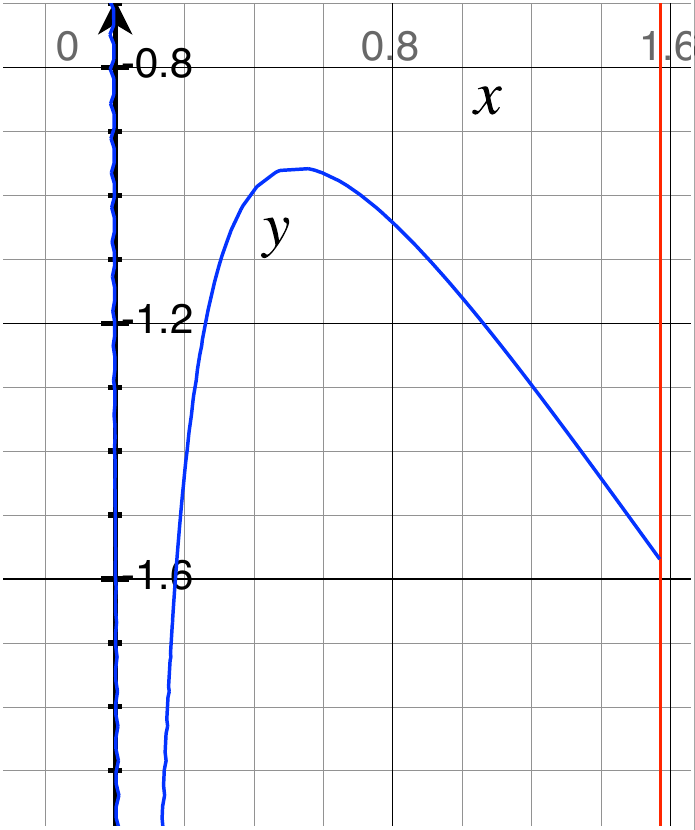}} }
\end{center}
\end{figure}

Notes that in \textbf{Case 3} the function $f(a)$ can be defined at $a=\f{\pi}{2}$, where the fiber of $p$ has a closed end sitting at the rightmost red vertical line. While in other cases we have open ends at both $0$ and $\f{\pi}{2}$, so the cohomology of compact support only counts for $t$ values of convex part. It is clear that \textbf{Case 2} and \textbf{Case 3} are actually continuous degenerations of \textbf{Case 1} to the left and right respectively. Thus we can synthesize definitions of $f(a)$ and $a_1$,$a_2$ in all cases and by  the definition of  $\Sigma=\{(\mathbf{q}_1,\mathbf{q}_2,t)|(\mathbf{q}_1,\mathbf{q}_2)\in D\, ,\, f(a_2)\leq t<f(a_1)\}$, we have the following isomorphisms :
$$
\Gamma\cong Rp_!\K_W\cong \K_\Sigma.
$$

\end{proof}

When $E$ is one dimensional, we can visualize $\Sigma$ in the following pictures. The vertical axis represents $t$-coordinate. The top is open while the bottom boundary surface is closed.\\

\begin{figure}[ht] 
\begin{center}
\mbox{
\subfigure[From above]{\includegraphics[height= 4.4cm]{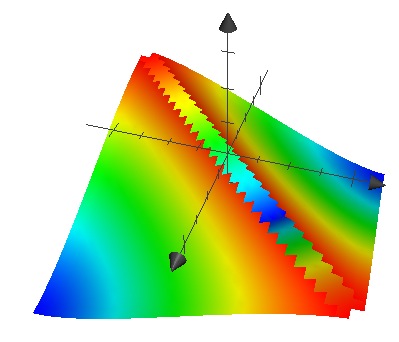}}
\subfigure[side]{\includegraphics[height= 4.4cm]{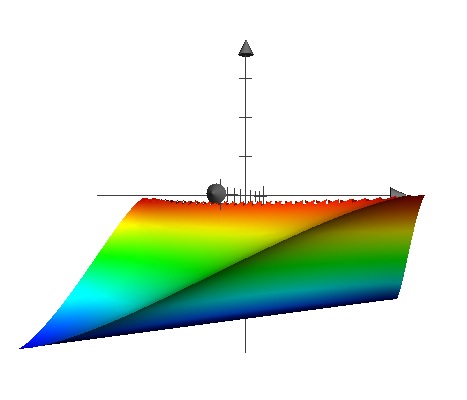}}
\subfigure[From below]{\includegraphics[height= 4.4cm]{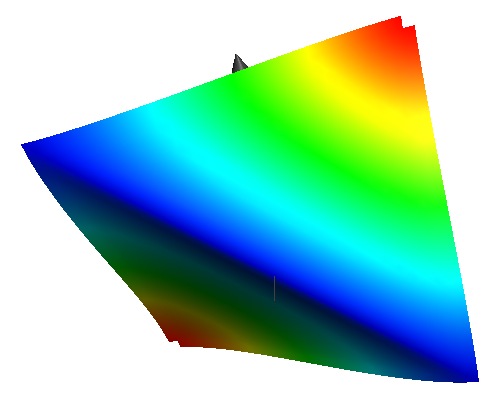}} }
\end{center}
\end{figure}

Recall the distinguished triangle (\ref{Gamma})
$$
\Gamma\rightarrow \K_A[n]\U{E}{\circ}\mathfrak{T}(\pr)\rightarrow\pr\xrightarrow{+1}.
$$

The building blocks of $\pr$ consist of copies of $\Gamma$ as follows. Let $\Gamma_1=\mathfrak{T}^{-1}\Gamma$ and for $j\geq2$ let $\Gamma_{j+1}=\mathfrak{T}^{-1}\K_A[-n]\U{E}{\circ}\Gamma_j$. Each constant sheaf $\Gamma_{j+1}$ is obtained from the previous constant sheaf $\Gamma_j$ by lifting its $t$-coordinate by $\f{\pi}{2}R^2$ and twisting $E_1\x E_2$ part and then shift its cohomological degree by $n$. The support of $\pr$ is a stack of the supports of those building blocks $\Gamma_j$'s. In fact one can glue all $\Gamma_j$ (concentrated at degree $jn$) to obtain the projector $\pr$ in the sense of homotopy colimit. 

Let us mention a few more words about the micro-support of $\Gamma$. $SS(\Gamma)$ has nontrivial cotangent fibers only for those $t$ hitting the critical value of $f(a)$. From the expression $f(a)=-S_a(\mathbf{q}_1,\mathbf{q}_2)-aR^2$ we have
$$
df=-dS-d(aR^2)=-\mathcal{A}^*\alpha+\alpha+Hda-R^2da
$$ 
$$
=\mathbf{p}_1d\mathbf{q}_1-\mathbf{p}_2d\mathbf{q}_2+(H-R^2)da=\f{\p f}{\p \mathbf{q}_1}d\mathbf{q}_1+\f{\p f}{\p \mathbf{q}_2}d\mathbf{q}_2+\f{\p f}{\p a}da.
$$

After evaluating at the critical points of $f(a)$, we get $H=R^2$, and hence 
$$
\mathbf{q}_i^2+(\f{\p f}{\p \mathbf{q}_2})^2=\mathbf{q}_i^2+\mathbf{p}_i^2=H=R^2.
$$

This also allows one to show the projector property (\textbf{Theorem \ref{decomposition}}) by concretely considering each blocks and their gluing. Notice that the upper cap can be deformed to the lower bottom by some non-characteristic deformations.\\

\subsection{Objects Micro-Supported on a Contact Ball}\label{contactsupp}\ \\

Now we can define the \textsl{contact} projector $\PR$ by lifting the symplectic framework. Recall that by \textbf{Theorem \ref{decomposition}} there exists a  symplectic projector which is  represented by the object $\pr$ (\ref{P}) of $\D(E_1\x E_2 \x \R_t)$, in the sense that for any object $\F$ of $\D(E_1\x\R_t)$ we have 
$$
\F\U{E_1}{\bu}\pr \in \mathcal{D}_B(E_2\x\R_t).
$$

Consider the difference map $\delta:\R_1\x\R_2\rightarrow\R_t$ , $(z_1,z_2)\mapsto z_2-z_1$. $\delta$ lifts kernel convolutions to compositions. Let us define
$$
\PR := \delta^{-1} \pr \in D(E_1\x\R_1\x E_2 \x\R_2)=D(X_1\x X_2).
$$

\begin{prop}\label{contactpr}
$\PR$ is the composition kernel of the projector from $\D(X)$ to $\mathcal{D}_{C_R}(X)$ (see \textbf{Definition \ref{DC}}).
\end{prop}

\begin{proof}
For any object $\F\in\D(X_1)=\D(E_1\x\R_1)$ we have the following isomorphisms by \textbf{Proposition \ref{conv-comp}}:
$$
\F\U{X_1}{\circ}\PR=\F\U{X_1}{\circ}(\delta^{-1}\pr)\cong\F\U{E_1}{\bu}\pr.
$$ 

This shows that composition with $\PR$ sends objects in $\D(X)$ to objects in $\mathcal{D}_{C_R}(X)$. On the other hand, let $\mathscr{Q}_R$ be the pull-back $\delta^{-1}\mathcal{Q}_R$ and $\widetilde{\Delta}$ be the set $\{(\mathbf{q}_1,z_1,\mathbf{q}_2,z_2)|\mathbf{q}_1=\mathbf{q}_2\, , z_2\geq z_1 \}$ in $X_1\x X_2$. Then by applying the functor $\delta^{-1}$ on this distinguished triangle
$$
\pr\rightarrow\K_\Delta\rightarrow\mathcal{Q}_R\xrightarrow{+1}
$$

we get the following distinguished triangle:
$$
\PR\rightarrow \K_{ \widetilde{\Delta}}\rightarrow \mathscr{Q}_R \xrightarrow{+1}.
$$

Thus the object $\PR$ serves as projector from $\D(X)$ to $\mathcal{D}_{C_R}(X)$. 
\end{proof}

In fact, $\PR$ lives in $\mathcal{D}_{<0,>0}(X_1\x X_2)$. The triangulated category $\mathcal{D}_{<0,>0}(X_1\x X_2)$ can be identified with the quotient of $D(X_1 \x X_2)$ by the full subcategory of objects micro-supported on the closed cone  $T^*(\R_1\x\R_2)\setminus\{\zeta_1<0,\zeta_2>0\}$. In other words, an object $\G$ lives in $\mathcal{D}_{<0,>0}(X_1\x X_2)$ if and only if the natural morphism $\G\rightarrow \G\ast\K_{\{z_1\leq0,\,z_2\geq0\}}$ is an isomorphism. This is analogue to \textbf{Proposition \ref{*}} and can be generalized to closed proper cones in the cotangent fibers.

We would also like to remark that the composition with an object of $\mathcal{D}_{<0,>0}(X\x X)$ gives rise to an endofunctor of the category $\D(X)$. One might think symbolically that the kernels in $\mathcal{D}_{<0,>0}$ eats the input  $>0$ with its $<0$ and leaves another $>0$ as the output.

\vspace{2cm}

\section{Contact Isotopy Invariants}\label{invariant}
\vspace{0.8cm}

\subsection{Admissible Open Subsets in the Prequantization Space.}\ \\

In this section we  define a family of contact isotopy invariants for certain open subsets $U$ in the prequantization space $\R^{2n}\x \mathds{S}^1$. First, lift $U$ to an open set $\widetilde{U}$ in $\R^{2n}\x\R_z$. Here $\widetilde{U}$ is periodic in the sense that it satisfies $\mathbf{T}\widetilde{U}=\widetilde{U}$ where $\mathbf{T} : z\mapsto z+1$. Let $X$ denote $\R^n\x\R_z$ and denote the cotangent coordinate of $z$ by $\zeta$. Consider the conification of $\widetilde{U}$ in $T^*_{\zeta>0}(X)$, namely the conic set $C_U:=\{(\mathbf{q,p},z,\zeta)|\zeta>0,(\mathbf{q,p}/\zeta,z)\in\widetilde{U}\}$. 
$$
\begin{array}{clr}
\,C_U \subset& T^*_{\zeta>0}(X)  \\
\downarrow  &\qquad\downarrow  \\
\quad\widetilde{U}\subset & \R^{2n}\x\R  \\
\downarrow  &\qquad\downarrow  \\
\quad U \subset & \R^{2n}\x\mathds{S}^1
\end{array}.
$$
Let $D_{Y\setminus C_U}(X)\hookrightarrow \D(X)$ be the full subcategory of objects micro-supported outside $C_U$. 

\begin{defn}\label{adm}
We call the open set $U$ \textbf{admissible} if the above embedding has a left adjoint functor, and if there exists a projector kernel $\PU\in D(X\x X)$ whose composition with $\D(X)$ projects to the quotient category $\D(X)/D_{Y\setminus C_U}(X)\cong\mathcal{D}_{C_U}(X)$.
\end{defn}

\noindent\textbf{Example.} The contact ball $B_R\x\mathds{S}^1$ is admissible by \textbf{Proposition \ref{contactpr}}. The associated projector kernel is $\PR$.\\

We need to show that the admissibility of an open subset is preserved under Hamiltonian contactomorphism. Let $\Phi$ be a Hamiltonian contactomorphism with compact support. This means that $\Phi$ is given by a contact isotopy (with compact support)  $\Phi_s : \R^{2n}\x \mathds{S}^1\rightarrow\R^{2n}\x \mathds{S}^1$ where $s\in I$ and $I$ is an interval containing $[0,1]$, $\Phi_0=Id$ and $\Phi_1=\Phi$. Recall that in \textbf{Section \ref{tricat}} we lift those $\Phi_s$ as in the following diagram:
$$
\begin{array}{clr}
C_U \subset& T^*_{\zeta>0}(X)  \xrightarrow{\Phi_s}T^*_{\zeta>0}(X)\\
\downarrow  &\qquad\downarrow  \qquad\qquad\quad \downarrow \\
\widetilde{U}\subset & \R^{2n}\x\R \xrightarrow{\Phi_s} \R^{2n}\x\R \\
\downarrow  &\qquad\downarrow  \qquad\qquad \quad \downarrow \\
U \subset & \R^{2n}\x\mathds{S}^1\xrightarrow{\Phi_s}\R^{2n}\x\mathds{S}^1
\end{array}.
$$

So the lifting of $\Phi$ becomes a homogeneous Hamiltonian symplectomorphism of $ T^*_{\zeta>0}(X) $. To see how $\Phi$ comes into play in the triangulated category of sheaves on $ T^*_{\zeta>0}(X) $, we need to look for its sheaf quantization. Recall that we have defined the overall Lagrangian graph (\ref{overall})
$$
\Lambda = \{(s,-h_s(\Phi_s(y)),y^\mathbf{a},\Phi_s(y)) | s\in I,y\in Y^\mathbf{a}\} \subset T^*(I\x X\x X).
$$ 

Roughly speaking, the sheaf quantization of $\Phi$ is a sheaf micro-supported in $\Lambda$. In the papers of Guillermou-Kashiwara-Schapira \cite{GuKaSch} and Guillermou \cite{Gu}, the existence and uniqueness of sheaf quantization have been well-established. We restate their result here:

\begin{thm}\label{quantization}(\cite{GuKaSch} Proposition 3.2)
There exists an unique locally bounded sheaf $\mathscr{S}$ in $D(I\x X\x X) $ satisfying 
$$
\mathbf{(1)}  \qquad SS(\mathscr{S})\subset \Lambda \cup T^*_{I\x X\x X}(I\x X\x X).
$$
$$
\mathbf{(2)} \quad\forall s\in I,\, \mathscr{S}_s\circ{\mathscr{S}^{-1}_s}\cong{\mathscr{S}^{-1}_s}\circ\mathscr{S}_s\cong \mathscr{S}_0\cong \K_{\Delta_X}.
$$

Here $\mathscr{S}_s$ denotes the pull-back of $\mathscr{S}$ to $\{s\}\x X\x X$ and  $\mathscr{S}^{-1}_s$ is defined to be the object  \, $v^{-1}\underline{Rhom}(\mathscr{S}_s\, ;\,\omega_X \boxtimes \K_X)$, $v:(x,y)\mapsto(y,x)$. In \cite{Gu}(Theorem 16.3) it is proved that we can make $ SS(\mathscr{S})=\Lambda $ outside the zero-section.
\end{thm}

\begin{rmk}\label{shiftSS}
 In Guillermou \cite{Gu} the term sheaf quantization is named by the fact that, outside the zero-section we have $SS(\F\circ\mathscr{S}_s)=\Phi_s(SS(\F))$ for any sheaf $\F$ on $X$. 
\end{rmk}

\begin{rmk} The sheaf $\mathscr{S}$ has a representative in the left semi-orthogonal piece    $\mathcal{D}_{<0,>0}(I\x X\x X)$. 
\end{rmk}

With this sheaf quantization $\mathscr{S}$ we can characterize admissibility under Hamiltonian contactomorphism $\Phi$. 

\begin{prop}\label{conj} If  $U$ is admissible then $\Phi(U)$ is admissible as well. In fact, for any $s\in I$ we can set $\mathscr{P}_{\Phi_s(U)}\cong {\mathscr{S}^{-1}_s}\circ\PU\circ \mathscr{S}_s$.
\end{prop}
\begin{proof}
First, by \textbf{Definition \ref{adm}} for any $\F\in\D(X)$ we have $\F\circ\mathscr{Q}_U\in D_{Y\setminus C_U}(X)$. By further composition with $\mathscr{S}_s$ and \textbf{Remark \ref{shiftSS}} we have 
$$
\F\circ(\mathscr{Q}_U\circ\mathscr{S}_s)\cong(\F\circ\mathscr{Q}_U)\circ\mathscr{S}_s\in D_{Y\setminus \Phi_s(C_U)}(X).
$$

Since composition with ${\mathscr{S}^{-1}_s}$ presents an auto-equivalence of $\D(X)$, we have the  composed functor
$$
\D(X)\overset{\circ{\mathscr{S}^{-1}_s}}{\cong}\D(X)\xrightarrow{\circ (\mathscr{Q}_U\circ\mathscr{S}_s)} D_{Y\setminus \Phi_s(C_U)}(X).
$$

Second, for any object $\F\in\D(X)$ and $\G\in D_{Y\setminus \Phi_s(C_U)}(X)$, again by \textbf{Remark \ref{shiftSS}} we have $\G\circ{\mathscr{S}^{-1}_s}\in D_{Y\setminus C_U}(X)$. So 
$$
Rhom(\F\circ(\PU\circ\mathscr{S}_s) ; \G)
$$
$$
\cong Rhom((\F\circ\PU)\circ\mathscr{S}_s ; \G\circ({\mathscr{S}^{-1}_s}\circ\mathscr{S}_s))
$$
$$
\cong Rhom((\F\circ\PU)\circ\mathscr{S}_s ; (\G\circ{\mathscr{S}^{-1}_s})\circ\mathscr{S}_s)
$$
$$
\overset{\circ{\mathscr{S}^{-1}_s}  }{\rightarrow} Rhom(  \F\circ\PU ; \G\circ{\mathscr{S}^{-1}_s}).
$$

On the other hand we have quasi-inverse morphism:
$$
Rhom(\F\circ\PU ; \G\circ{\mathscr{S}^{-1}_s} )  \overset{\circ\mathscr{S}_{s}  }{\rightarrow}{Rhom((\F\circ\PU)\circ\mathscr{S}_s ; (\G\circ{\mathscr{S}^{-1}_s})\circ\mathscr{S}_s)}.
$$

Thus by the projector property of $\PU$ we get
$$
Rhom(\F\circ(\PU\circ\mathscr{S}_s) ; \G)\cong Rhom(  \F\circ\PU ; \G\circ{\mathscr{S}^{-1}_s})\cong 0.
$$

So we obtain a functor to the left semi-orthogonal complement 
$$
\D(X)\overset{\circ{\mathscr{S}^{-1}_s}}{\cong}\D(X)\xrightarrow{\circ (\PU\circ\mathscr{S}_s)}\quad    ^\perp(D_{Y\setminus \Phi_s(C_U)}(X)).
$$

Similarly, the composition with ${\mathscr{S}^{-1}_s}\circ\mathscr{Q}_U\circ \mathscr{S}_s $ gives a functor to the right semi-orthogonal complement 
$$
\D(X)\rightarrow (D_{\Phi_s(C_U)}(X))^\perp.
$$

Finally, from the distinguished triangle with respect to $U$
$$
\PU\rightarrow\K_{\widetilde{\Delta}}\rightarrow \mathscr{Q}_U\xrightarrow{+1}.
$$

It is easy to see that after conjugation we have another triangle:
$$
{\mathscr{S}^{-1}_s}\circ\PU\circ \mathscr{S}_s\rightarrow ({\mathscr{S}^{-1}_s}\circ\K_{\widetilde{\Delta}}\circ\mathscr{S}_s\cong\K_{\widetilde{\Delta}})\rightarrow{\mathscr{S}^{-1}_s}\circ\mathscr{Q}_U\circ \mathscr{S}_s \xrightarrow{+1}.
$$

As a result, the kernel $\mathscr{P}_{\Phi_s(U)}:={\mathscr{S}^{-1}_s}\circ\PU\circ \mathscr{S}_s$ represents the projector with respect to the (hence admissible) open set $\Phi_s(U)$. Here is an illustration diagram

\xymatrix{   
 &&\D(X) \ar[d]_{\PU} \ar@/^/[r]^{\mathscr{S}_s} &\D(X) \ar[d]^{\mathscr{P}_{\Phi_s(U)}} \ar@/^/[l]^{{\mathscr{S}^{-1}_s}} \\
&&\mathcal{D}_{C_U}(X) \ar[r]_{\mathscr{S}_s}  & \mathcal{D}_{\Phi_s(C_U)}(X)   \ar@{=}[r] & \mathcal{D}_{C_{\Phi_s(U)}}(X).
}

\end{proof}

\subsection{Cyclic Actions and Contact Invariants}\label{cyclic}\ \\

Pick a positive integer $N$ and consider the cyclic action on the space $(X\x X)^N$, 
$$
\sigma : (x_1,x_2,\cdots,x_{2N-1},x_{2N})\mapsto (x_{2N},x_1,x_2,\cdots,x_{2N-1}).
$$

\noindent\textbf{Assumption.} From now on, we set the integer $N$ to be a prime number and let the ground field $\K$ be the finite field of $N$ elements.\\

This subsection is aimed to stress the necessity of cyclic group actions associated to $\sigma:(x_1,\cdots,x_{2N-1},x_{2N})\mapsto(x_{2N},x_1,\cdots,x_{2N-1})$. Let me explain the motivation here. In order to homogenize contact isotopy $\{\Phi_s\}$ on $\R^{2n}\x\mathds{S}^1$, we lift the admissible open set $U\subset\R^{2n}\x\mathds{S}^1$ to its covering $\widetilde{U}\subset \R^{2n}\x\R$  so that it becomes a quotient of $T^*_{\zeta>0}(X)$ where $X$ denotes $\R^{2n}\x\R_z$. Then the machinery of micro-support comes into play and defines the projector $\PU\in D(X\x X)$. What has been done so far only allows $\PU$ to see the real line $\R_z$ but not the circle $\mathds{S}^1$. The idea is to pick a large enough integer $N$ to approximate $\mathds{S}^1$ by the finite cyclic group $\sfrac{\Z}{N\Z}$. Here the infinitesimal circle action is approximated by the cyclic generator $\sigma^2$ of $\sfrac{\Z}{N\Z}$. 

To be precise we need to adopt the terminology of \textsl{equivariant derived category}. For the general theory of equivariant derived categories we refer the reader to the book by Bernstein and Lunts \cite{BL}. Let $\mathcal{G}$ be a group acting on a topological space $Z$. Consider the following diagram of spaces

\begin{multicols}{2}

\xymatrix{
\qquad\mathcal{G}\x\mathcal{G}\x Z \ar@/^1.4pc/[r]^{d_0} \ar[r]_{d_1}  \ar@/_1.4pc/[r]_{d_2} & \mathcal{G}\x Z\ar@/^1pc/[r]^{d_0} \ar@/_1pc/[r]_{d_1} & Z \ar[l]_{s_0} 
}

$ 
\begin{cases} 
s_0(z)=(e,z)\\
d_0(g_1,\cdots,g_n,z)=(g_2,\cdots,g_n,g_1^{-1}z)\\
d_i(g_1,\cdots,g_n,z)=(g_1,\cdots,g_ig_{i+1},\cdots,g_n,z)\\
d_n(g_1,\cdots,g_n,z)=(g_1,\cdots,g_{n-1},z).
\end{cases}
$
\end{multicols}

A $\mathcal{G}$-equivariant sheaf on $Z$ is a pair $(\F,\tau)$ where $\F\in Sh(Z)$ and $\tau:d^*_1\F\cong d^*_0\F$ is an isomorphism satisfying the cocycle conditions $d^*_2\tau\circ d^*_0\tau=d^*_1\tau$ and $s^*_0\tau=id_{\F}$. Equivariant sheaves form an abelian category $Sh_{\mathcal{G}}(Z)$. Now let $\mathcal{G}$ be the cyclic group $\sfrac{\Z}{N\Z}$ acting on the space $Z=(X\x X)^N$. Since $\sfrac{\Z}{N\Z}$ is finite, it follows that $Sh_{\sfrac{\Z}{N\Z}}((X\x X)^N)$ has enough injectives and its derived category is equivalent to the bounded below $\sfrac{\Z}{N\Z}$-equivariant derived category $D^+_{\sfrac{\Z}{N\Z}}((X\x X)^N)$.

Let $\mathscr{T}$ be the constant sheaf supported on the shifted diagonal, namely 
$$
\mathscr{T}:=\K_{\{\mathbf{q}_1=\mathbf{q}_2, \, z_1-z_2=1\}} \in D(X\x X).
$$ 
Consider the object $\sigma_*\mathscr{T}^{\boxtimes N}$. It is just the constant sheaf supported on the set $\{\q_{2j}=\q_{2j+1}, z_{2j}-z_{2j+1}=1,j\in\sfrac{\Z}{N\Z}\}$, which is obviously an $\sfrac{\Z}{N\Z}$-equivariant sheaf.  On the other hand, the object $\mathscr{P}_{U}$ is a complex(no matter which complex we choose from its class) on which the group $\sfrac{\Z}{N\Z}$ acts trivially, thus the object $\mathscr{P}_{U}^{\boxtimes N}$ can be represented by a complex consists of objects of $Sh_{\sfrac{\Z}{N\Z}}((X\x X)^N)$.

We now formulate the invariant as follows:

\begin{defn} Given an admissible open set $U$ and a prime integer $N$. The contact isotopy invariant of $U$ is an object of $D^+_{\sfrac{\Z}{N\Z}}(pt)$ defined by 
$$
\C(U):= {Rhom}_{(X\x X)^N}(\PU^{\boxtimes N};\sigma_*\mathscr{T}^{\boxtimes N}).
$$
\end{defn}

Here the functor $Rhom$ is taken over the category $D^+_{\sfrac{\Z}{N\Z}}((X\x X)^N)$. Let $B(\sfrac{\Z}{N\Z})$ be the classifying space of the group $\sfrac{\Z}{N\Z}$. The equivariant derived category $D^+_{\sfrac{\Z}{N\Z}}(pt)$ is equivalent to the full subcategory of $D^+(B(\sfrac{\Z}{N\Z}))$ consisting of all complexes with constant cohomology sheaves of $\K$-vector spaces.  One may also define $\C(U)$ as an object of the derived category of the modules over the group algebra $\K[\sfrac{\Z}{N\Z}]$ and consider its "quotient" in  $D^+_{\sfrac{\Z}{N\Z}}(pt)$.

We claim that $\C$ is indeed a contact isotopy invariant for admissible open subsets of $\R^{2n}\x\mathds{S}^1$. 
Now we can state the invariance property under contact isotopies.

\begin{thm}\label{inv} \textbf{(1)} For any admissible open set $U$ and any $s\in I$ there is an isomorphism between invariants $\C(U)\cong\C(\Phi_s(U))$, and \textbf{(2)} any embedding $V\overset{i}{\hookrightarrow} U$ of one admissible into another induces an morphism $\C(U)\xrightarrow{i^*}\C(V)$ naturally.
\end{thm}
\begin{proof}
\textbf{Proof of (1).} We first investigate the relationship between external tensor product and composition operation. Let $\F$ and $\G$ be objects in $D(X_1\x X_2)$. For $1\leq j\leq N$, let $p_j: X_{2j-1}\x X_j'\x X_{2j}\rightarrow X_{2j-1}\x X_j'$, $q_j: X_{2j-1}\x X_j'\x X_{2j}\rightarrow  X_j'\x X_{2j}$ and $r_j : X_{2j-1}\x X_j'\x X_{2j}\rightarrow  X_{2j-1}\x X_{2j}$ be projections and let $\pi_j:\displaystyle\Pi^{N}_{j=1} (X_{2j-1}\x X_{2j})\rightarrow X_{2j-1}\x X_{2j}$ be projection to the $j$-th factor. Notes that all notations $X_j$ and $X_j'$ are just labels identical to the same space $X$. We have
$$
(\F\U{X}{\circ}\G)^{\boxtimes N}=\bigotimes_{j=1}^N \pi_j^{-1} r_{j!}(p_j^{-1}\F\otimes q_j^{-1}\G)
$$
$$
\cong (\Pi_{j=1}^N r_j)_! [( \displaystyle\prod_{j=1}^N p_j  )^{-1}\F^{\boxtimes N} \otimes ( \displaystyle\prod_{j=1}^N q_j  )^{-1}\G^{\boxtimes N}]
\cong \F^{\boxtimes N}  \U{\Pi X_j'}{\circ} \G^{\boxtimes N}.
$$

So for $U$ admissible we can set ${\mathscr{S}^{-1}_s}\in D(X_1\x X')$, $\PU\in D(X'\x X'')$ and $\mathscr{S}_s\in D(X''\x X_2)$, then in $D^+_{\sfrac{\Z}{N\Z}}((X_1\x X_2)^N)$ we have
$$
\mathscr{P}_{\Phi_s(U)}^{\boxtimes N} = ({\mathscr{S}^{-1}_s}\U{X'}{\circ}\PU\U{X''}{\circ}\mathscr{S}_s)^{\boxtimes N}\cong {\mathscr{S}^{-1}_s}^{\boxtimes N}\U{\Pi X_j'}{\circ}\PU ^{\boxtimes N} \U{\Pi X_j''}{\circ}\mathscr{S}_s^{\boxtimes N}.
$$

Hence 
\begin{equation}\label{Eugene}
Rhom(\mathscr{P}_{\Phi_s(U)}^{\boxtimes N};\sigma_*\mathscr{T}^{\boxtimes N})\cong Rhom({\mathscr{S}^{-1}_s}^{\boxtimes N}\U{\Pi X_j'}{\circ}\PU ^{\boxtimes N}\U{\Pi X_j''}{\circ}\mathscr{S}_s^{\boxtimes N};  \sigma_*\mathscr{T}^{\boxtimes N} )
\end{equation}
$$
\cong Rhom(  \PU ^{\boxtimes N};  \mathscr{S}_s^{\boxtimes N}\U{\Pi X_{2j-1}}{\circ} (\sigma_*\mathscr{T}^{\boxtimes N} )\U{\Pi X_{2j}}{\circ} {\mathscr{S}^{-1}_s}^{\boxtimes N} ),
$$
here $\mathscr{S}_s\in D(X'\x X_1)$ and ${\mathscr{S}^{-1}_s}\in D(X_2\x X'')$. Note that $\sigma_*\mathscr{T}^{\boxtimes N} $ is a constant sheaf supported by the set
$$
\{\mathbf{q}_{2j}=\mathbf{q}_{2j+1},\, z_{2j}-z_{2j+1}= 1,\, j\in \sfrac{\Z}{N\Z}  \}
$$ 
and by uniqueness of sheaf quantization (\textbf{Theorem \ref{quantization}}), we know that $\mathscr{S}$ is $\Z$-equivariant on its $\R_z$-component, so we have
\begin{equation}\label{Eric}
\mathscr{S}_s^{\boxtimes N}\U{\Pi X_{2j-1}}{\circ}  (\sigma_*\mathscr{T}^{\boxtimes N})\cong (\sigma_*\mathscr{T}^{\boxtimes N})\U{\Pi X_j''}{\circ}  \mathscr{S}_s^{\boxtimes N} \, ,
\end{equation}

(\ref{Eugene}) and (\ref{Eric}) induce an isomorphism
$$
Rhom(  \PU ^{\boxtimes N};  \mathscr{S}_s^{\boxtimes N}\circ (\sigma_*\mathscr{T}^{\boxtimes N} )\circ {\mathscr{S}^{-1}_s}^{\boxtimes N} )\cong Rhom(  \PU ^{\boxtimes N};\sigma_*\mathscr{T}^{\boxtimes N} ),
$$
which reads
$$
\C(\Phi_s(U))\cong  \C(U).
$$
\\
\noindent\textbf{Proof of (2).} Let $V\overset{i}{\hookrightarrow} U$ be two admissible open subsets in $\R^{2n}\x \mathds{S}^1$. Consider the distinguished triangle in $\mathcal{D}_{<0,>0}(X\x X)$:
$$
\mathscr{P}_U \rightarrow  \K_{ \widetilde{\Delta}} \rightarrow\mathscr{Q}_U\xrightarrow{+1}
$$
and apply "$\circ\mathscr{P}_V$" on it, we get 
$$
\mathscr{P}_U\circ\mathscr{P}_V\rightarrow\mathscr{P}_V\rightarrow\mathscr{Q}_U\circ\mathscr{P}_V\xrightarrow{+1}.
$$
\indent Since $V\subset U$, we have $D_{Y\setminus C_U}(X)\subset D_{Y\setminus C_V}(X)$. Hence $\mathscr{Q}_U\circ\mathscr{P}_V\cong0$, and the above distinguished triangle reads
\begin{equation}\label{UVV}
\mathscr{P}_U\circ\mathscr{P}_V\cong\mathscr{P}_V.
\end{equation}
\indent On the other hand, by admissibility of $V$ we have the morphism $\mathscr{P}_V\rightarrow \K_{ \widetilde{\Delta}}$. After composing with $\mathscr{P}_U$ we get 
\begin{equation}\label{UVU}
\mathscr{P}_U\circ\mathscr{P}_V\rightarrow \mathscr{P}_U.
\end{equation}
By (\ref{UVV})(\ref{UVU}) there is a morphism $i:\mathscr{P}_V\cong \mathscr{P}_U\circ\mathscr{P}_V\rightarrow \mathscr{P}_U$. It is clear that this morphism extends to $\mathscr{P}_V^{\boxtimes N}\xrightarrow{i^{\boxtimes N}} \mathscr{P}_U^{\boxtimes N}$, hence gives rise to a morphism in $D^+_{\sfrac{\Z}{N\Z}}((X\x X)^N)$:
$$
\C(U)=Rhom(\PU^{\boxtimes N};\sigma_*\mathscr{T}^{\boxtimes N})\xrightarrow{i^*} Rhom(\mathscr{P}_V^{\boxtimes N};\sigma_*\mathscr{T}^{\boxtimes N})=\C(V).
$$

Naturality comes from the following commutative diagram between projectors:  let $W\overset{j}{\hookrightarrow} V\overset{i}{\hookrightarrow}  U$ be admissible open subsets,\\

\xymatrix{
&\mathscr{P}_W \ar[r]^{\cong} & \mathscr{P}_V\circ\mathscr{P}_W \ar[r]^{j}  \ar[d]^{\cong} &\mathscr{P}_V\cong \PU\circ\mathscr{P}_V \ar[r]^{\qquad i} & \mathscr{P}_U\\
&&(\PU\circ\mathscr{P}_V)\circ\mathscr{P}_W\ar[r]^{\cong} &\PU\circ(\mathscr{P}_V\circ\mathscr{P}_W)\ar[r]^{\quad\cong} & \PU\circ \mathscr{P}_W\ar[u]^{i\circ j}.
}  
\quad\\

We obtain the desired commutative diagram
\xymatrix{
&\C(U)\ar[r]^{i^*} \ar[dr]_{(i\circ j)^*}  & \C(V) \ar[d]^{j^*}\\
&& \C(W).
}

\end{proof}

When $U$ is the contact ball $B_R\x\mathds{S}^1$, the corresponding projector $\mathscr{P}_R$ encodes Hamiltonian rotations. We will see later that it gives a good enough approximation to the Hamiltonian loop space in $\R^{n}$ after we unwrap the invariant $\C(B_R\x\mathds{S}^1)$. Here the cyclic action shifts the time parameter of the loop.

\vspace{1cm}

\subsection{Invariants for Contact Balls}\ \\

In this subsection we make an computational approach to the contact isotopy invariants $\C(U)$ for  $U=\R^{2n}\x\mathds{S}^1$ and $U=B_R\x\mathds{S}^1$ . Let $\bar{\delta}:\Pi_{j=1}^N E_{2j-1}\x E_{2j}\x \R_{2j-1}\x\R_{2j}\rightarrow \Pi_{j=1}^N E_{2j-1}\x E_{2j}\x \R_{t_j}$ be the difference map on each pair $(z_{2j-1},z_{2j})\xrightarrow{\bar{\delta}} z_{2j}- z_{2j-1}$ . Let us begin with
$$
\sigma_*\mathscr{T}^{\boxtimes N}=\K_{\{\mathbf{q}_{2j}=\mathbf{q}_{2j+1},\, z_{2j}-z_{2j+1}= 1,\, j\in \sfrac{\Z}{N\Z}  \}}.
$$
 
We introduce new variables $t_j=z_{2j}-z_{2j-1}$. If $z_{2j}-z_{2j+1}= 1$ for each $j$, we have $\Sigma t_j = \Sigma (z_{2j}-z_{2j-1)}= \Sigma (z_{2j}-z_{2j+1)}=N$. Conversely, any  $t_j$'s satisfing $\Sigma t_j=N$ can be rewritten in the form $t_j=z_{2j}-z_{2j-1}$ such that $z_{2j}-z_{2j+1}= 1$ for each $j$. Hence we have the following sequence of isomorphisms  in $D((E\x E\x\R)^N)$ after applying $\bar{\delta}$ on $\sigma_*\mathscr{T}^{\boxtimes N}$ :
\begin{equation}\label{Peter}
\bar{\delta}_*\sigma_*\mathscr{T}^{\boxtimes N}=\K_{\{\mathbf{q}_{2j}=\mathbf{q}_{2j+1},\,\Sigma t_j= N,\, j\in \sfrac{\Z}{N\Z}  \}}
\end{equation}
$$
=\bar{\Delta}_* \K_{E^N\x \{\Sigma t_j=N\}}
$$
$$
= \bar{\Delta}_*{\pi_\Sigma}^{-1} \K_{E^N \x \{t=N \}}
$$
$$
\cong \bar{\Delta}_*{\pi_\Sigma}^{!} \K_{E^N \x \{t=N \}}[1-N]
$$
$$
= \bar{\Delta}_*{\pi_\Sigma}^{!}  {\pi}^{-1}\K_{t=N} [1-N]
$$
$$
\cong \bar{\Delta}_*{\pi_\Sigma}^{!}  {\pi}^{!}\K_{t=N} [-nN] [1-N],
$$
here $\pi:E^N\x\R\rightarrow\R$ is the projection along $E^N$, $\pi_\Sigma:E^N\x\R^N\rightarrow E^N\x\R$  takes summation of all $t_j$ to the new variable $t$, and $\bar{\Delta}$ is the (shifted)diagonal embedding of $E^N$ into $E^{2N}$ defined by $E_j \overset{\Delta}\hookrightarrow E_{2j-2}\x E_{2j-1} $ for each $j\in \sfrac{\Z}{N\Z} $ .\\

Therefore for $U=\R^{2n}\x\mathds{S}^1$ we have 

\begin{lem}\label{whole}
$\C(\R^{2n}\x\mathds{S}^1)\cong\K[(1+n)(1-N)]$ in $D(\K[\sfrac{\Z}{N\Z}])$. Here $\K[\sfrac{\Z}{N\Z}]$ acts on $\K$ trivially.
\end{lem}

\begin{proof} 
By (\ref{Peter}),
$\C(\R^{2n}\x\mathds{S}^1)
=Rhom( {\K_{\widetilde{\Delta}}}^{\boxtimes N};\sigma_*\mathscr{T}^{\boxtimes N}  )\cong Rhom( \bar{\delta}^{-1} {\K_{\Delta}}^{\boxtimes N};\sigma_*\mathscr{T}^{\boxtimes N}  )
$
$$
\cong Rhom ({\K_{\Delta}}^{\boxtimes N};  \bar{\delta}_*\sigma_*\mathscr{T}^{\boxtimes N}) 
$$
$$
\cong Rhom ({\K_{\Delta}}^{\boxtimes N};  \bar{\Delta}_* {\pi_\Sigma}^! {\pi}^! \K_{t= N}  [1-(n+1)N])
$$
$$
\cong Rhom(  R\pi_{!}  R\pi_{\Sigma!}  \bar{\Delta}^{-1} ({\K_{\Delta}}^{\boxtimes N} ) ;\K_{t= N}[1-(n+1)N])
$$
$$
\cong Rhom ( \K_{t\geq0}[-n] ; \K_{t= N} [1-(n+1)N])  
$$
$$
\cong\K[(1+n)(1-N)].
$$

It is easy to see that the cyclic group acts trivially on both $\K_{t\geq0}[-n]$ and $\K_{t= N} [1-(n+1)N]$  and hence gives trivial action on $\C(\R^{2n}\x\mathds{S}^1)$.

\end{proof}

When $U$ is the contact ball $B_R\x\mathds{S}^1$ we replace $\K_{\widetilde{\Delta}}$ by $\mathscr{P}_{R}$ and the same adjunctions gives
$$
\C(B_R\x\mathds{S}^1)=Rhom(\mathscr{P}_{R}^{\boxtimes N};\sigma_*\mathscr{T}^{\boxtimes N} )
$$
$$
\cong Rhom(R\pi_{!} R\pi_{\Sigma!}\bar{\Delta}^{-1}(\pr^{\boxtimes N});\K_{t= N} [1-(n+1)N] ).
$$

Define 
\begin{equation}\label{Paul}
\F_R:=R\pi_{!} R\pi_{\Sigma!}\bar{\Delta}^{-1}(\pr^{\boxtimes N})
\end{equation} 
then we can rewrite $\C(B_R\x\mathds{S}^1)\cong Rhom(\F_R;\K_{t= N} [1-(n+1)N])$. Now we are going to unwrap the stalks of $\F_R$ by computing cohomology with compact supports of certain free loop space with an cyclic action on it.

\begin{lem}\label{fr}
Given $T\geq0$, choose an integer $M$ such that $M>\f{4T}{\pi N R^2}$. For $A\in G$ satisfying $0\leq -A \leq \f{T}{NR^2}$ and $\mathbf{q}_{j,k}\in E^{NM}$, define the quadratic form $\Omega(A) = \sum_{j = 1}^N \sum_{k = 1}^M S_{\sfrac{A}{M}}(\mathbf{q}_{j,k},\,\mathbf{q}_{j,k+1})$. Define $\mathcal{W}=\{(A,\{\mathbf{q}_{j,k}\})| \Omega(A)\geq0  \}$ a subset of $G\x E^{NM}$. Denote $\rho:G\x E^{NM}\rightarrow G$ the projection onto $G$ and let $\mathcal{E}= R\rho_! \K_\mathcal{W}$. One has 
$$
\F_R|_T \cong R\Gamma_c(\{0\leq -A \leq \f{T}{NR^2}\} ; \mathcal{E}).
$$
The cyclic action on the set $\mathcal{W}$ is given by $\mathbf{q}_{i,j}\mapsto\mathbf{q}_{i-1,j}$. The cyclic action on $\mathcal{E}$ is induced from $\mathcal{W}$ by computing the equivariant cohomology of the pre-image of $\rho$.
\end{lem}

\begin{proof}
First, let us unwrap $\pr$ in terms of the sheaf quantization $\s$ and its convolutions (or compositions). By (\ref{P}) we have
$$
\pr:=\widehat{\s}\U{b}{\circ}\K_{\{b<R^2\}}=(\s\U{a}{\bu}\K_{\{t+ab\geq0\}}[1])\U{b}{\circ}\K_{\{b<R^2\}}
$$
$$
\cong\s\U{a}{\bu}(\K_{\{t+ab\geq0\}}\U{b}{\circ}\K_{\{b<R^2\}})[1]\cong\s\U{a}{\bu}\K_{\{ (a,t)|a\leq 0 \, ,\, t+aR^2\geq 0\}}.
$$

To see $\pr^{\boxtimes N}$, for each $j\in\sfrac{\Z}{N\Z}$ let $\pi_j$ be the projection from $(E\x E)^N$ to $E_{2j-1}\x E_{2j}$ and let $\s_j\in\D(G_j\x E_{2j-1}\x E_{2j}\x\R_{t_j})$ be the sheaf quantization of Hamiltonian rotations. Also let $\pi_G$ be the projection along $G^N$ and $\bar{s}$ takes addition $\R_t^N\x\R_t^N\rightarrow \R_t^N$. We have the following isomorphisms:
$$
\pr^{\boxtimes N}\cong(\bigotimes_{j=1}^N \pi_j^{-1}\s_j)\U{G^N}{\bu}\K_{\{ \forall j,\, \sfrac{-t_j}{R^2}\leq a_j\leq0 \}}\cong R\pi_{G!}R\bar{s}_! [(\bigotimes_{j=1}^N \pi_j^{-1}\s_j)\boxtimes \K_{\{ \forall j,\, \sfrac{-t_j}{R^2}\leq a_j\leq0 \}}].
$$

Furthermore let $\R_t\x\R_t\xrightarrow{s}\R_t$ denote the summation. Therefore we can write $\F_R$ (\ref{Paul}) as 
\begin{equation}\label{Patrick}
\F_R:=R\pi_{!} R\pi_{\Sigma!}\bar{\Delta}^{-1}(\pr^{\boxtimes N})\cong R\pi_! R\pi_{G!} Rs_![(R\pi_{\Sigma !}(\bar{\Delta}^{-1} \bigotimes_{j=1}^N \pi_j^{-1}\s_j))\boxtimes \K_{\{\forall j, \, a_j\leq0,\, t\geq (-\Sigma a_j)R^2\}} ].
\end{equation}

Let me describe the object $R{\pi_\Sigma}_!\bar{\Delta}^{-1}{\bigotimes_{j=1}^N} \pi_j^{-1}\s_j$. Given any tuple $\vec{a}=(a_1,\cdots,a_j,\cdots,a_N)$ in $ G^N$ in which each $a_j$ is non-positive, there exists an positive integer $M$ such that for any $j$ we have $-\f{\pi}{4}<\f{a_j}{M}\leq0$. We can divide $[a_j,0]$ into subintervals of length less than $\f{a_j}{M}$ by setting $[a_j,0]=\bigcup_{k=1}^M[\f{k}{M}a_j,\f{k-1}{M}a_j]$. Recall that on each subinterval $[\f{k}{M}a_j,\f{k-1}{M}a_j] $ the object $\s_j$ has the expression in terms of the following local generating function (\ref{GF})
$$
S_a(\mathbf{q,q'})=\f{1}{2\tan(2a)}(\mathbf{q}^2+\mathbf{q'}^2)-\f{1}{\sin(2a)}\mathbf{qq'}.
$$

Iterated convolutions of (\ref{explicit form}) from $1$ to $M$ gives us
\begin{equation}\label{Peru}
\s_j|_{a_j}=R\rho_{j!}\K\{(\mathbf{q}_{j,1},\cdots,\mathbf{q}_{j,k},\cdots,\mathbf{q}_{j,M+1},t_j)|\, t_j +\sum_{k  = 1}^M  S_{\sfrac{a_j}{M}}(\mathbf{q}_{j,k},\mathbf{q}_{j,k+1})\geq0 \}
\end{equation}
here $\rho_{j}$ denotes the projection: $(\mathbf{q}_{j,1},\cdots,\mathbf{q}_{j,k},\cdots,\mathbf{q}_{j,M+1},t_j)\mapsto ( \mathbf{q}_{j,1},\mathbf{q}_{j,M+1},t_j)$, 
and we identify the variables $\mathbf{q}_{j,1}=\mathbf{q}_{2j-1}$ and $\mathbf{q}_{j,M+1}=\mathbf{q}_{2j}$ such that $\s_j|_{a_j}\in\D(E_{2j-1}\x E_{2j}\x\R_{t_j}  )$. 

Observe that the morphism $\bar{\Delta}^{-1}$ gives further identification $\mathbf{q}_{2j}=\mathbf{q}_{2j+1}$(i.e., $\mathbf{q}_{j,M+1}=\mathbf{q}_{j+1,1}$) for all $j\in\sfrac{\Z}{N\Z}$. Denote 
\begin{equation}\label{phoenix}
\mathcal{W}(\vec{a}) =\{(\{\mathbf{q}_{j,k}\},t)|\, t+ \displaystyle\sum_{j = 1}^N \sum_{k = 1}^M    S_{\sfrac{a_j}{M}}(\mathbf{q}_{j,k},\,\mathbf{q}_{j,k+1})\geq0 \}
\end{equation}
the subset of $E^{NM}\x\R_t$, and $\rho_E$ denotes the projection: $E^{NM}\rightarrow E^N$ mapping from $\{\mathbf{q}_{j,k}\}$ to $\{\mathbf{q}_j\}$.  We see that the functor $R\pi_{\Sigma !}$ is given by $\sum_{j=1}^N$ in the definition of $\mathcal{W}(\vec{a})$. Therefore by (\ref{Peru}) and (\ref{phoenix}) we get
\begin{equation}\label{Pakistan}
R{\pi_\Sigma}_!\bar{\Delta}^{-1}{\bigotimes_{j=1}^N} \pi_j^{-1}(\s_j|_{a_j})\cong R\rho_{E!}\K_{\mathcal{W}(\vec{a})}.
\end{equation}

\begin{figure}[ht] 
\includegraphics[height=8cm]{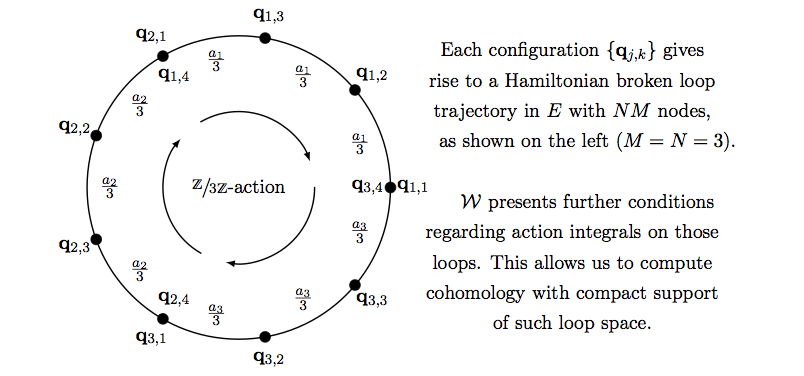}
\end{figure}

Therefore pointwisely $R\pi_{\Sigma !}(\bar{\Delta}^{-1} \bigotimes_{j=1}^N \pi_j^{-1}\s_j)$ describes the concatenation of sheaf quantization of Hamiltonian rotations. Again by the argument around  (\ref{explicit form}) the sheaf quantization should only depend on $\sum a_j$. Namely, let $\alpha_\Sigma : G^N \rightarrow G$ be the summation $\vec{a}\mapsto \Sigma a_j$, then the object $(R\pi_{\Sigma !}(\bar{\Delta}^{-1} \bigotimes_{j=1}^N \pi_j^{-1}\s_j))$ is a pull-back by $\alpha_\Sigma$. So there is an object $\mathscr{G}\in D_{\sfrac{\Z}{N\Z}}(G\x E^N\x\R_t)$ such that 
\begin{equation}\label{Pisa}
R\pi_{\Sigma !}(\bar{\Delta}^{-1} \bigotimes_{j=1}^N \pi_j^{-1}\s_j)\cong\alpha_\Sigma^{-1} \mathscr{G}.
\end{equation}

We introduce a new variable $A:=\sfrac{\Sigma a_j}{N}$. Denote the subset 
$$
\mathcal{W}(A)=\{t+ \displaystyle\sum_{j = 1}^N \sum_{k = 1}^M    S_{\sfrac{A}{M}}(\mathbf{q}_{j,k},\,\mathbf{q}_{j,k+1})\geq0 \}\subset E^{NM}\x \R.
$$
By (\ref{Pakistan}) and (\ref{Pisa}), pointwisely we have 
\begin{equation}\label{Pandora}
\mathscr{G}|_A\cong R\rho_{E!} \K_{\mathcal{W}(A)}.
\end{equation} 

Let $\pi_A$ be the projection along $A$. We can decompose the projection $\pi_G$ into the composition of $\pi_A$ and $\alpha_\Sigma$. Plug those expressions into $\F_R$ (\ref{Patrick}) we get  
\begin{equation}\label{Paris}
\F_R\cong R\pi_! R\pi_{G!} Rs_![(R\pi_{\Sigma !}(\bar{\Delta}^{-1} \bigotimes_{j=1}^N \pi_j^{-1}\s_j))\boxtimes \K_{\{\forall j, \, a_j\leq0,\, t\geq (-\Sigma a_j)R^2\}} ]
\end{equation}
$$
\cong R\pi_!R\pi_{A!}  Rs_!R\alpha_{\Sigma !} [\alpha_\Sigma^{-1}(\mathscr{G})\boxtimes  \K_{\{\forall j, \, a_j\leq0,\, t\geq (-\Sigma a_j)R^2\}}]
$$
$$
\cong R\pi_!R\pi_{A!}  Rs_! [\mathscr{G} \boxtimes   R\alpha_{\Sigma !} \K_{\{\forall j, \, a_j\leq0,\, t\geq (-\Sigma a_j)R^2\}}]
$$
$$
\cong R\pi_!R\pi_{A!}  Rs_! [\mathscr{G} \boxtimes   \K_{\{t_2\geq -NAR^2\geq 0\}}].
$$

To unwrap $\F_R$ to more explicit form, we write $\mathscr{G}\in D_{\sfrac{\Z}{N\Z}}(G\x E^N\x \R_1)$ and $ \K_{\{t_2\geq -NAR^2\geq 0\}}\in \D(G\x\R_2)$ and the summation $s: \R_1\x\R_2\rightarrow \R_t$. Consider the projection along $E^N$ by $\pi_E$ $:G\x E^N\x\R_1\rightarrow G\x\R_1$, we claim there exist $\mathcal{E}\in D_{\sfrac{\Z}{N\Z}}(G)$ and $\mathcal{K}\in \D(\R_1)$ such that 
\begin{equation}\label{Panama}
R\pi_{E!}\mathscr{G}\cong \mathcal{E} \boxtimes \mathcal{K}.
\end{equation}

 To be more precise, let us fix any $A=\sfrac{\Sigma a_j}{N}$ and recall (\ref{Pandora}) that we have $\mathscr{G}|_A\cong R\rho_{E!} \K_{\mathcal{W}(A)}\in\D(E^N\x\R_1)$ where $\mathcal{W}(A)=\{t_1+ \displaystyle\sum_{j = 1}^N \sum_{k = 1}^M    S_{\sfrac{A}{M}}(\mathbf{q}_{j,k},\,\mathbf{q}_{j,k+1})\geq0 \}\subset E^{NM}\x \R_1$ and $\rho_E$ denotes the projection: $E^{NM}\rightarrow E^N$ mapping from $\{\mathbf{q}_{j,k}\}$ to $\{\mathbf{q}_j\}$. Let 
 \begin{equation}\label{Omega} 
\Omega(A) = \displaystyle\sum_{j = 1}^N \sum_{k = 1}^M S_{\sfrac{A}{M}}(\mathbf{q}_{j,k},\,\mathbf{q}_{j,k+1})
 \end{equation}
be the homogeneous quadratic form on $E^{NM}$, we see that the only critical value of $-\Omega(A)$ is zero. Hence $R\pi_{E!}\mathscr{G}|_{A}\cong R\pi_{E!}R\rho_{E!} \K_{\mathcal{W}(A)}$ is supported by points greater or equal to zero, which is independent of the choice of $A$. This shows the existence of the factor $\mathcal{E}$, and we can set another factor $\mathcal{K}$ to be the constant sheaf  $\K_{\{t_1\geq0\}}$. 

Let $\rho_t:G\x E^{NM}\x\R\xrightarrow{\rho_E} G\x E^N\x\R \xrightarrow{\pi_E} G\x\R$. Since (\ref{Pandora}) $\mathscr{G}|_A\cong R\rho_{E!}\K_{\mathcal{W}(A)}\cong R\pi_{E!}R\rho_{E!} \K_{\mathcal{W}(A)}$,  pointwisely we have
$$
\mathcal{E}|_A\cong R\pi_{E!}R\rho_{E!} \K_{\mathcal{W}(A)}|_{t=0}= R\rho_{t!}\K_{\mathcal{W}(A)}|_{t=0}.
$$

Denote $\rho=\rho_t|_{t=0}:G\x E^{NM}\rightarrow G$. Define $\mathcal{W}=\{(A,\{\mathbf{q}_{j,k}\})| \Omega(A)\geq0  \}$ a subset of $G\x E^{NM}$. Then we have 
$$
\mathcal{E}\cong R\rho_! \K_\mathcal{W}.
$$

Let us continue on $\F_R$. Recall that $R\pi_{E!}\mathscr{G}\cong \mathcal{E} \boxtimes \K_{\{t_1\geq0\}}$. Therefore by (\ref{Paris}) and (\ref{Panama}) in $\D(\R_t)$ we have 
$$
\F_R\cong R\pi_!R\pi_{A!}  Rs_! [\mathscr{G} \boxtimes   \K_{\{t_2\geq -NAR^2\geq 0\}}]
\cong  R\pi_{A!} (\mathcal{E} \boxtimes \K_{\{t\geq -NAR^2\geq 0\}}). 
$$

Therefore for any given $t=T\geq0$ we get (here $R\Gamma$ computes $\sfrac{\Z}{N\Z}$-equivariant cohomology)
$$
\F_R|_T \cong R\Gamma_c(\{0\leq -A \leq \f{T}{NR^2}\} ; \mathcal{E}).
$$
When $T\geq0$ is given, there is a uniform integer $M$ such that $0\leq \f{-A}{M} < \f{\pi}{4}$, namely we can pick any integer $M>\f{4T}{\pi NR^2}$.
\end{proof}

Notice that from $\mathcal{E}\cong R\rho_! \K_\mathcal{W}$ we have 

\begin{equation}\label{Canada}
\F_R|_T \cong R\Gamma_c(\{0\leq -A \leq \f{T}{NR^2}\} ; \mathcal{E})\cong  R\Gamma_{c}(\rho^{-1}(\{0\leq -A \leq \f{T}{NR^2}\})\cap \mathcal{W}).
\end{equation}

This way we see that the invariant $\C(B_R\x\mathds{S}^1)\cong Rhom(\F_R;\K_{t= N} [1-(n+1)N])$ is given by computing compactly supported equivariant cohomology of the space $\rho^{-1}(\{0\leq -A \leq \f{T}{NR^2}\})\cap \mathcal{W}$. Moreover, we have 

\begin{lem}\label{FR}
$\F_R|_T$ is the compactly supported equivariant cohomology of the space $\R^{D+1}$ where $\R^{D+1}\hookrightarrow E^{NM}$ and $D=2n\lceil\f{T}{\pi R^2}\rceil-n-1$. The cyclic group action on $\R^{D+1}$ is induced from the action on $E^{NM}=E\x\mathds{C}^{\f{NM-1}{2}n}$ by complex multiplication by the exponents of $\omega^M$ on the complex coordinates respectively. Here $\omega$ stands for the primitive $MN$-th  root of unity.
\end{lem}
\begin{proof}
Relabel the index $\{\q_{j,k}\}=\{\q_\ell\}=\vec{\q}$ in lexicographical order and $\ell\in\sfrac{\Z}{MN\Z}$. The condition $\Omega(A)\geq0$ on $\mathcal{W} $ translates to (\ref{Omega}) $\sum_{\ell} S_{\sfrac{A}{M}}(\q_\ell,\q_{\ell+1})\geq0$.  In other words, the points $(A,\vec{\q})\in\mathcal{W}$ are characterized by the inequality:
$$
\cos(\f{2A}{M})(\sum_\ell {\q_\ell}^2)\leq (\sum_\ell\q_\ell\q_{\ell+1}).
$$

By the assumption of $M$, which is $|\sfrac{2A}{M}|<\sfrac{\pi}{2}$, on $\mathcal{W}$ we have
$$
0\leq \cos(\f{2A}{M})(\sum_\ell{\q_\ell}^2)\leq (\sum_\ell \q_\ell\q_{\ell+1}).
$$
When $\vec{\q}$ has identical components,  for $A$ the condition (\ref{Canada}) $\rho^{-1}(\{0\leq -A \leq \f{T}{NR^2}\})\cap \mathcal{W}$ becomes the closed interval $0\leq -A \leq \f{T}{NR^2}$. On the other hand, by the fact $|\f{\sum_\ell\q_\ell\q_{\ell+1}}{\sum_\ell {\q_\ell}^2}|\leq 1$, other $\vec{\q}$ give another closed interval constraint on $A:$
$$
\arccos(\f{\sum_\ell\q_\ell\q_{\ell+1}}{\sum_\ell {\q_\ell}^2})\leq -\f{2A}{M}\leq \f{2T}{MNR^2}\,(\leq \f{\pi}{2}).
$$

In any of the above cases,  the constraint on $A$ has compactly supported cohomology $(\cdots\rightarrow\K\rightarrow\cdots)$ where $\K$ is located at degree $0$. 

On the other hand, to unwrap the condition $\Omega(A)\geq0$ of $\mathcal{W}$ we need to diagonalize the  matrix of quadratic form $\Omega(A)$(\ref{Omega}). Let $\omega=\exp(\f{2\pi i}{MN})$ be the generator of $MN$-th  roots of unity. Here we can take $MN$ to be an odd integer. The set of eigenvalues of $\Omega(A)$ is 
$$
\{\lambda_\ell=\cot(2A/M)-\f{1}{2}\csc(2A/M)(\omega^\ell+\omega^{NM-\ell})| \ell\in \sfrac{\Z}{MN\Z}\}
$$
or
$$
\{\lambda_\ell=\cot(2A/M)-\csc(2A/M)\cos(\f{2\pi }{MN}\ell)| \ell\in \sfrac{\Z}{MN\Z}\}.
$$

Among them the eigenvalue $\lambda_0$ corresponds to eigenvector $(1,1,\cdots,1)$. For each $\ell\neq0$, set  $\vec{\theta}_\ell=(1,\omega^\ell,\omega^{2\ell},\cdots)$ and denote its real and imaginary part by $\vec{\theta}_\ell = \vec{R}_\ell + i\vec{I}_\ell$. Then the eigenspace of $\lambda_\ell (=\lambda_{MN-\ell})$ is orthogonally spanned by $\vec{R}_\ell$ and $\vec{I}_\ell$. Let $\xi_0=\langle(1,1,\cdots,1),\vec{\mathbf{q}}\rangle$ and $u_\ell=\langle\vec{R}_\ell,\vec{\mathbf{q}}\rangle$ and $v_\ell=\langle\vec{I}_\ell,\vec{\q}\rangle$. Therefore the condition $\Omega(A)\geq0$ becomes 
$$
\lambda_0\xi_0^2+\displaystyle\sum_{\ell=1}^{\sfrac{(MN-1)}{2}}\lambda_\ell (u_\ell^2 + v_\ell^2)\geq 0.
$$

Hence the constraint (\ref{Canada}) $\rho^{-1}(\{0\leq -A \leq \f{T}{NR^2}\})\cap \mathcal{W}$ is homeomorphic to a Euclidean space $\R^{D+1}$. Here $+1$ stands for homogeneity of the quadratic form $\Omega$. Thus we can write
$$
\F_R|_T \cong  R\Gamma_c(\rho^{-1}(\{0\leq -A \leq \f{T}{NR^2}\})\cap \mathcal{W})\cong R\Gamma_c(\R^{D+1}).
$$

The generator of the cyclic action is given by $\mathbf{q}_j\mapsto\mathbf{q}_{j-M}$ and induces 
$$
\sum_{j\in\sfrac{\Z}{N\Z}} \mathbf{q}_j\omega^{j\ell}\mapsto\sum_{j\in\sfrac{\Z}{N\Z}} \mathbf{q}_{j-M}\omega^{j\ell}=\omega^{M\ell}\sum_{j\in\sfrac{\Z}{N\Z}} \mathbf{q}_j\omega^{j\ell}.
$$
Therefore the cyclic action on $E^{NM}$ is generated by  
\begin{equation}\label{ell}
u_\ell+i v_\ell \mapsto \omega^{M\ell} (u_\ell+i v_\ell).
\end{equation}

To compute $D$, let $\mathcal{I}$ be the number of positive eigenvalues of $\Omega$ among $\ell=1,\cdots,\f{MN-1}{2}$. After projecting along $A$ we can replace $-A$ by $\f{T}{NR^2}$, then $\mathcal{I}$ becomes the number of solutions of $\ell$ of  
\begin{equation}
\lambda_\ell=\cot(-\f{2T}{MNR^2})-\csc(-\f{2T}{MNR^2})\cos(\f{2\pi }{MN}\ell) > 0.
\end{equation}

Therefore   $\mathcal{I}=\lceil\f{T}{\pi R^2}\rceil-1$. On the other hand, from the fact $\F_R|_T\cong R\Gamma_c(\R^{D+1})$ we know that  $D+1=n+2n\mathcal{I}$. Here doubling $\mathcal{I}$ comes from pairing up $u_\ell$ and $v_\ell$, and $n=\dim(E)$. Therefore
\begin{equation}\label{Chile}
D=n+2n\mathcal{I}-1=2n\lceil\f{T}{\pi R^2}\rceil-n-1.
\end{equation}

\end{proof}

The cyclic group acts on $\F_R|_T $ by its action on $\R^{D+1}$. In this case one may take $T=N$ and write $\C(B_R\x\mathds{S}^1)\cong Rhom(\K_{t=N}[-D-1];\K_{t= N} [1-(n+1)N])\cong\K[D+2-(n+1)N]$ in $D(\K[\sfrac{\Z}{N\Z}])$.  

\begin{cor}\label{R=infty}
$\F_\infty |_T \cong R\Gamma_c(E) \cong\K[-n]$ as a $\K[\sfrac{\Z}{N\Z}]$-module.
\end{cor}
\begin{proof}
Also note that $\R^{2n}\x\mathds{S}^1$ corresponds to the limit $R\rightarrow\infty$. In this case we must have $A=0$ and the inequality $|\f{\sum_\ell\q_\ell\q_{\ell+1}}{\sum_\ell {\q_\ell}^2}|\leq 1$ should hold its equality and forces $\vec{\q}$ to have identical components parametrized by $E$. Therefore the cyclic group acts trivially on it and as a $\K[\sfrac{\Z}{N\Z}]$-module we get
$$
\F_\infty |_T \cong R\Gamma_c(\{A=0\} ; \mathcal{E})\cong R\Gamma_c(E) \cong\K[-n].
$$ 
\end{proof}

This recovers \textbf{Lemma \ref{whole}}. 
\vspace{1cm}
\subsection{The Cones of Embeddings}\ \\

For the purpose of testing squeezability in the ambient space $\R^{2n}\x\mathds{S}^1$, instead of individual $C_N(B_R\x\mathds{S}^1)$ we need to focus on the cone of the morphism $C_N(\R^{2n}\x\mathds{S}^1)\xrightarrow{j_R^*} C_N(B_R\x\mathds{S}^1)$. By the definition (\ref{Paul}) of $\F_R$ this is the cone of 
$$
Rhom(\F_\infty;\K_{t= N} [1-(n+1)N] )\rightarrow Rhom(\F_R;\K_{t= N} [1-(n+1)N] ).
$$

\begin{prop}\label{D}
Let $\mathscr{L}_R$ be the cocone of the morphism $\F_R\rightarrow\F_\infty$ (which means that $\mathscr{L}_R[1]$ is the cone). Then in $D(\K[\sfrac{\Z}{N\Z}])$ one has $\mathscr{L}_R\cong R\Gamma_c(\mathds{S}^{D-n}\x E\x \R_{>0})$. Here $D=2n\lceil\f{T}{\pi R^2}\rceil-n+1$.  The cyclic group acts on $\mathds{S}^{D-n}$ by complex multiplication (\ref{ell}).
\end{prop}

\begin{proof}

\textbf{Lemma \ref{fr}}, and \textbf{Corollary \ref{R=infty}} tell us 
$$\mathscr{L}_R|_T\cong R\Gamma_c(\{0< -A \leq \f{T}{NR^2}\} ; \mathcal{E})\cong R\Gamma_c(\rho^{-1}(\{0< -A \leq \f{T}{NR^2}\})\cap \mathcal{W}).
$$

Following \textbf{Lemma \ref{FR}} it is easy to see that the constraint set is homeomorphic to $R^{D+1}$ with all diagonals $(\q,\q,\q,\cdots,\q)$ removed. This is homeomorphic to $\mathds{S}^{D-n} \x E\x\R_{>0}$. 

Recall that to $E^{NM-1}$ we assign complex coordinate $u_\ell+i v_\ell\in \mathds{C}^n$ for $\ell=1,\cdots,\frac{MN-1}{2}$. The $\sfrac{\Z}{N\Z}$ group action on $\mathds{S}^{D-n}$ part is given by (\ref{ell}) $u_\ell+i v_\ell \mapsto \omega^{M\ell} (u_\ell+i v_\ell)$ for those $\ell$ corresponding to positive eigenvalues. The $\sfrac{\Z}{N\Z}$ actions are trivial on both $E$ and $\R_{>0}$ parts.

\end{proof}

Recall that $\C(B_R\x\mathds{S}^1)\cong Rhom(\F_R ; \K_{t= N} [1-(n+1)N] )$. We should take $T=N$ so that in this case we have $D=2n\lceil\f{N}{\pi R^2}\rceil-n+1$. In the rest of the paper we are only interested in the case $\pi R^2\geq1$. Then the initial assumption of $M>\f{4T}{\pi NR^2}$ in \textbf{Lemma \ref{fr}}
 can be simplified to

\vspace{0.3cm}
\noindent\textbf{Assumption.} We assume $T=M=N>3 $ is a prime number.
\vspace{0.3cm}

Recall that $\mathcal{I}$ is the number of positive eigenvalues $\lambda_\ell$ among  $\ell=1,\cdots,\frac{N^2-1}{2}$. Let $\zeta=\omega^M=\omega^N$ be the primitive $N$-th root of the unity. The cyclic action on $\mathcal{L}_R|_{T=N}$ is given by (\ref{ell}) $u_\ell+i v_\ell \mapsto \zeta^{\ell} (u_\ell+i v_\ell)$ for $\ell=1,\cdots,\mathcal{I}$. Therefore by (\ref{Chile}) we have

\begin{cor}
If $\pi R^2\geq 1$ then the group $\sfrac{\Z}{N\Z}$ acts freely on the sphere $\mathds{S}^{D-n}$.
\end{cor}
\begin{proof}
This is because $\pi R^2\geq1$ implies   $\mathcal{I}=\lceil\f{N}{\pi R^2}\rceil-1\leq N-1$.
\end{proof}

Recall that $\mathscr{L}_R|_{T=N}$ computes the compactly supported equivariant cohomology of $\mathds{S}^{D-n}\x E\x \R_{>0}$. From the above corollary, we see the $\sfrac{\Z}{N\Z}$-quotient of this sphere $\mathds{S}^{D-n}=\mathds{S}^{2n\mathcal{I}}$ is the lens space $L(N;1,\cdots,1,2,\cdots,2,\cdots,\mathcal{I},\cdots,\mathcal{I})$. Note that the group $\sfrac{\Z}{N\Z}$ acts trivially on the $E\x \R_{>0}$ part. So we have

\begin{cor}\label{free}
Assume that $\pi R^2\geq1$. In the equivariant derived category $\mathcal{L}_R|_{T=N}[n+1]$ becomes the cohomology of (D-n)-dimensional lens space.
\end{cor}

We will apply this fact to the proof of non-squeezability.

\vspace{2cm}

\section{Proof of Non-Squeezability}\label{proof}
\vspace{0.75cm}
Now given $1\leq\pi r^2<\pi R^2$. We are going to prove contact non-squeezability by contradiction. Assume that $B_R\x\mathds{S}^1$ can be squeezed into $B_R\x\mathds{S}^1$ by $\Phi$, and $\Phi$ can be joint to the identity via compactly supported contact isotopy.

By \textbf{Theorem \ref{inv}} on this sequence of open embeddings $\Phi(B_R\x\mathds{S}^1)\overset{j}{\hookrightarrow}B_r\x\mathds{S}^1\overset{j_r}{\hookrightarrow} \R^{2n}\x\mathds{S}^1 $ we have the following commutative diagram for induced morphisms:

\xymatrix{
&& \C( \Phi(B_R\x\mathds{S}^1)) \ar[d]_{\cong} & \C(B_r\x\mathds{S}^1)\ar[l]_{\quad j^*}   &  \C(  \R^{2n}\x\mathds{S}^1)\ar[l]_{{j_r}^*} \ar[lld]^{{j_R}^*} \\
&& \C(B_R\x\mathds{S}^1).
}

Recall \textbf{Proposistion} \ref{D} that there is a distinguished triangle
$$
\C(\R^{2n}\x\mathds{S}^1)\xrightarrow{j_R^*}\C(B_R\x\mathds{S}^1)\rightarrow Rhom(\mathscr{L}_R;\K_{t=N}[1-(n+1)N])\xrightarrow{+1}.
$$

Before working in the equivariant setting, let us shift $\mathscr{L}_R$ by $n+1$ to insure they all live in non-negative degrees:
$$
\C(\R^{2n}\x\mathds{S}^1)\xrightarrow{j_R^*}\C(B_R\x\mathds{S}^1)\rightarrow Rhom(\mathscr{L}_R[n+1];\K_{t=N}[n+2-(n+1)N])\xrightarrow{+1}.
$$

We need to consider the cyclic action on the cone of the morphism $j_R^*$. In the category of $D(\K[\sfrac{\Z}{N\Z}])$ one has $\C(\R^{2n}\x\mathds{S}^1)\cong\K[D+1-(n+1)N]$ and $\C(B_R\x\mathds{S}^1)\cong \K[D+2-(n+1)D]$. The group algebra $\K[\sfrac{\Z}{N\Z}]$ acts on the former $\K$ by zero. To distinguish, we write $\C(\R^{2n}\x\mathds{S}^1)\cong\K_1[D+1-(n+1)N]$ and $\C(B_R\x\mathds{S}^1)\cong \K_2[D+2-(n+1)D]$ where $\K[\sfrac{\Z}{N\Z}]$ acts trivially on $\K_1$. We write $j_R^*\in Rhom_{\K[\sfrac{\Z}{N\Z}]}(\K_1,\K_2)[D-n+1]$.  $N$ is an odd prime and $\K$ is the finite field of $N$ elements.

In order to discuss the equivariant cohomology of $cone(j_R^*)$, let us pass to the cohomology of $Rhom^{\bu}_{\K[\sfrac{\Z}{N\Z}]}(-,-)$. It is the Yoneda product $Ext_{\K[\sfrac{\Z}{N\Z}]}^\bu(-;-)$. The product of cohomology classes $[j^*]\circ[j_r^*]=[j_R^*]$ is given by concatenation 
$$
Ext^\bu(\K_2;\K_2)\x Ext^\bu(\K_1;\K_2)\rightarrow Ext^{\bu+\bu}(\K_1;\K_2).
$$  

It turns out that the size of the contact ball has a crucial cohomological interpretation using this $Ext^\bu$ structure:

\begin{lem}\label{nonzero}
If $\pi R^2\geq1$ then $[j_R^*]$ is nonzero in $Ext_{\K[\sfrac{\Z}{N\Z}]}^{D-n+1}(\K_1;\K_2)$.
\end{lem}
\begin{proof}
If the cohomology class of $j_R^*$ becomes zero in $Ext_{\K[\sfrac{\Z}{N\Z}]}^{D-n+1}(\K_1;\K_2)$, then up to a degree shifting $cone(j_R^*)$ is quasi-isomorphic to $\K_1\bigoplus\K_2[D-n]$ in the derived category $D(\K[\sfrac{\Z}{N\Z}])$. It turns out that in $D_{\sfrac{\Z}{N\Z}}^{+}(pt)$ up to a degree shifting $cone(j_R^*)$ is just the equivariant cohomology of the disjoint union of a point $\{\ast\}$ and a $\R^{D-n}$. Since $\sfrac{\Z}{N\Z}$-action fixed $\{\ast\}$, we see that in $D_{\sfrac{\Z}{N\Z}}^{+}(pt)$ the object $cone(j_R^*)$ has unbounded cohomology.

On the other hand, we have $cone(j_R^*)\cong Rhom(\mathscr{L}_R[n+1];\K_{t=N}[n+2-(n+1)N])$. \textbf{Corollary \ref{free}} tells that the equivariant cohomology of $\mathscr{L}_R|_{T=N}[n+1]$ is the cohomology of a ($D-n$)-dimensional lens space. We see that in $D_{\sfrac{\Z}{N\Z}}^{+}(pt)$ the object $cone(j_R^*)$ should have bounded cohomology.

We arrive at the  contradiction by computing the equivariant cohomology of $cone(j_R^*)$ in two different ways. Therefore $j_R^*$ does pass to a nonzero class in $Ext_{\K[\sfrac{\Z}{N\Z}]}^\bu(\K_1;\K_2)$. And it makes sense to write $\deg([j_R^*])=D-n+1$.

\end{proof}

The non-squeezing property of contact balls in large scale is followed by the degree counting.

\begin{thm}
There is no such $j^*$ making the diagram commutative.
\end{thm}
\begin{proof}
For the diagram to commute we have $j_R^*=j^*\circ j_r^*$. By \textbf{Lemma \ref{nonzero}} we have $\deg([j_R^*])=\deg([j^*])+\deg([j_r^*])$. It follows from \textbf{Proposition \ref{D}} that 
$$
2n\lceil\f{N}{\pi R^2}\rceil-2n+2=\deg([j_R^*])\geq \deg([j_r^*])   =2n\lceil\f{N}{\pi r^2}\rceil-2n+2.
$$

On the other hand, according to the assumption $\pi r^2<\pi R^2$, we can pick large enough prime number $N$ (thanks to Euclid) such that  
$$
\lceil\f{N}{\pi R^2}\rceil < \lceil\f{N}{\pi r^2}\rceil   ,
$$
which leads to a contradiction. 

\end{proof}

As a corollary we conclude that \\

\textbf{Theorem \ref{main}}  \emph{ 
If $1\leq \pi r^2<\pi R^2$, then it is impossible to squeeze $B_R\x\mathds{S}^1$ into $B_r\x\mathds{S}^1$.
}

\vspace{2cm}

\section{Appendix} \label{appendix}

A bunch of relations between subsets are extensively applied in estimating microlocal singular supports of certain sheaves.  Their explanations and proofs can be found in several references(see \cite{GuKaSch}\cite{GuSch}\cite{KaSch}\cite{Ta}). Here we make a short list of those which take roles in \textbf{Section \ref{projector}}.

\vspace{1cm}

Let $Y\xrightarrow{f} X$ be morphism of manifolds. We have natural morphisms 
$$
T^*X\xleftarrow{f_\pi} T^*X\x_X Y\xrightarrow{f^t} T^*Y.
$$

Microlocal singular supports under Grothendieck's six functors($Rf_*,Rf_!,f^{-1},f^!,\otimes,\underline{Rhom}$) are bounded by the original ones, via $f^t$ and $f_\pi$, in the sense of the following three propositions:

\begin{prop}\label{push}( \cite{KaSch},Proposition 5.4.4)
Let $\G\in D(Y)$ and assume $f$ is proper on $supp(\G)$, then 
$$
SS(Rf_*\G)=SS(Rf_!\G)\subset f_\pi((f^t)^{-1}(SS(\G))).
$$
\end{prop}

\begin{prop}\label{pull}( \cite{KaSch},Proposition 5.4.13)
Let $\F\in D(X)$ and assume $f$ is non-characteristic for $SS(\F)$, i.e., $f_\pi^{-1} (SS(\F))\cap T_Y^*X\subset Y\x_X T_X^*X$. Then\\
\vspace{-0.4cm}

(i) $SS(f^{-1}\F)\subset(f^t)(f_\pi^{-1}(SS(\F)))$.

(ii) $ f^{-1}\F \otimes \omega_{Y/X} \cong f^!\F$.
\end{prop}

\begin{prop}\label{tenhom}( \cite{KaSch},Proposition 5.4.14) Let $\F,\G\in D(X)$. Then\\
(i)   $    SS(\F)\cap SS(\G)^a\subset T_X^*X\Rightarrow SS(\F\otimes\G)\subset SS(\F)+SS(\G).    $\\
(ii) $  SS(\F)\cap SS(\G)\subset T_X^*X\Rightarrow SS(\underline{Rhom}(\F;\G))\subset SS(\F)^a + SS(\G).$
\end{prop}

In addition, we need a more accurate control of microlocal singular supports under projection maps. Let $E$ be a nontrivial finite-dimensional real vector space and $p:X\x E\rightarrow X$ be the projection map. If an object $\F$ is nonsingular in a neighborhood of $\{(x,\omega)\}\x T_E^*E$ in $T^*(X \x E)$, then its projection along $E$ will remain nonsingular at $(x,\omega)$. 

\begin{prop}\label{key}( \cite{Ta},Corollary 3.4)
Let $\kappa:T^*X\x E\x E^\star\rightarrow T^*X\x E^\star$ be the projection and let $i:T^*X\rightarrow T^*X\x E^\star$ be the closed embedding $(x,\omega)\mapsto(x,\omega,0)$ of zero-section. Then for $\F\in D(X\x E)$ we have
$$
SS(Rp_!\F),SS(Rp_*\F)\subset i^{-1}\overline{\kappa(SS(\F))}.
$$
\end{prop}


\vspace{1cm}

\begin{thebibliography}{20}

\bibitem{AlMe}
  P. Albers and W. Merry,
  \emph{Orderability, contact non-squeezing, and Rabinowitz Floer homology}.
  \texttt{arXiv:1302.6576}
        
\bibitem{BL}
 J. Berstein and V. Lunts,
 \emph{Equivariant Sheaves and Functors}.
 Lecture Notes in Mathematics, vol.\textbf{1578}, Springer-Verlag(1994).
 
\bibitem{ChNe1}
  V. Chernov and S. Nemirovski,
  \emph{Non-negative Legendrian isotopy in $ST^*M$}.
  Geometry and Topology 14 (2010), 611-626 
 
\bibitem{ChNe2}
  V. Chernov and S. Nemirovski,
  \emph{Universal orderability of Legendrian isotopy classes}.
  \texttt{arXiv:1307.5694}
 

\bibitem{EKP}
 Y. Eliashberg, S. Kim and L. Polterovich,
 \emph{Geometry of contact transformations and domains:
orderability versus squeezing}. 
 Geometry and Topology 10 (2006) 1635-1747.
 
\bibitem{Fr} 
 M. Fraser,
 \emph{Contact non-squeezing at large scale in $\R^{2n} \times \mathds{S}^1$}.
  Personal Communication.
 
\bibitem{Gi}
  E. Giroux,
  \emph{Sur La G\'{e}om\'{e}trie et La Dynamique des Transformations de Contact  }.
  S\'{e}minaire BOURBAKI  61\`{e}me ann\'{e}e, 2008-2009, no 1004

\bibitem{Gu}
  S. Guillermou,
  \emph{Quantization of conic Lagrangian submanifolds of cotangent bundles}.
  \texttt{arXiv:1212.5818}
  
\bibitem{Gu2}
  S. Guillermou,
  \emph{The Gromov-Eliashberg theorem by microlocal sheaf theory}.
   \texttt{arXiv:1311.0187}

\bibitem{GuKaSch}
  S. Guillermou, M. Kashiwara and P. Schapira,
  \emph{Sheaf quantization of Hamiltonian isotopies and applications to non displaceability problems.}
  Duke Math. Journal Vol 161, pp.~201--245 (2012).

\bibitem{GuSch}
 S. Guillermou and P. Schapira,
  \emph{Microlocal theory of sheaves and Tamarkin's non displaceability theorem}.
  \texttt{arXiv:1106.1576}

\bibitem{KaSch}
 M. Kashiwara and P. Schapira,
 \emph{Sheaves on Manifolds}. 
  Grundlehren der Math. Wiss. \textbf{292} Springer-Verlag (1990).
 
\bibitem{Mc}
  D. McDuff and D. Salamon,
  \emph{Introduction to Symplectic Topology}.
  Oxford University Press (1998).
 
\bibitem{NaZa}
 D. Nadler and E. Zaslow, 
 \emph{Constructible sheaves and the Fukaya category}.
  J. Amer. Math. Soc. 22 (2009), 233-286.
 
 \bibitem{Sa}
 S. Sandon,
 \emph{Contact Homology, Capacity and Non-Squeezing in $\R^{2n}\x\mathds{S}^1$ via Generating Functions}.  
  Annales de l'institut Fourier, 61 no.1 (2011), p. 145-185.
  
\bibitem{Sa2}
 S. Sandon,
  \emph{An integer valued bi-invariant metric on the group of contactomorphisms of $\R^{2n} \x \mathds{S}^1$}.
  \texttt{arXiv:0910.5632}
  
\bibitem{Ta}
  D. Tamarkin,
  \emph{Microlocal Condition for Non-Displaceability}.
  \texttt{arXiv:0809.1584}

\bibitem{Tr}
  L. Traynor,
  \emph{Symplectic Homology via Generating Functions}.
  Geom. Funct. Anal. \textbf{4}(1994), 718-748.

\bibitem{Vi}
  C. Viterbo,
  \emph{An Introduction to Symplectic Topology through Sheaf theory}.\\
  \url{http://www.math.ens.fr/~viterbo/Eilenberg/Eilenberg.pdf}  
  


  
\end{thebibliography}
\end{document}